\documentclass[11pt]{amsart}
\usepackage{import}
\usepackage{style}
\usepackage[utf8]{inputenc}
\usepackage[english]{babel}
\usepackage{csquotes}
\usepackage[backend=biber,style=alphabetic]{biblatex} 
\makeatletter
\def\subsection{\@startsection{subsection}{2}%
  \z@{.5\linespacing\@plus.7\linespacing}
{.5\baselineskip}%
  {\normalfont\centering\scshape}%
}
\makeatother

\title{Phase transition for Loewner evolutions with complex linear drivers}
\author{Luis Brummet}
\email{Luis.Brummet@aalto.fi}
\address{Aalto University, Otakaari 1, 02150 Espoo, Finland.}

\begin{document}

\begin{abstract}
We study deterministic Loewner evolutions on the complex plane driven by complex-valued functions. This model can be viewed as a generalization of real-driven Loewner evolutions in the upper half-plane, or as the deterministic analogue of complex-driven Schramm-Loewner evolutions. First, we contribute to the already known theory of such evolutions. We establish a sufficient condition for drivers in the $C^1$-class to create a two-sided simple curve. By constructing a counterexample in the $C^0$-class, we demonstrate that the same condition is not necessary and discuss an alternative necessary and sufficient condition for $C^0$-drivers that create two-sided curves.

Second, we analyze the evolutions driven by the one-parameter family of complex linear drivers $\{ct\}_{c \in \mathbb C}$. We show that the geometries of the generated hulls differ significantly from the chordal real-driven case. Although each complex linear driver creates a two-sided curve, the geometry of the generated curve exhibits three distinct geometric phases depending on the complex parameter $c$: a simple phase, a simple with one end spiraling phase, and a third new exotic variant. In this exotic phase, one part of the curve is simple while the other part forms a Jordan curve rooted at the origin. After forming the Jordan curve, this part ceases to grow while disconnecting an open set of positive area from infinity for arbitrarily large times. We determine the phase boundaries in terms of $c$ via the signs of an explicit expression. Within the Hölder-$1/2$-class, we improve the upper bound on a constant sufficient to ensure that the driver creates a two-sided simple curve.
\end{abstract} 

\maketitle

\newpage

\thispagestyle{empty}
\counterwithin{thm}{section}
\counterwithin{equation}{section}
\counterwithin{example}{section}
\counterwithin{figure}{section}

\tableofcontents

\section{Introduction}

\subsection{Context and motivation}

The chordal Loewner differential equation is a first order differential equation with initial value problem given by
\[
\partial_t g_t(z) = \frac{2}{g_t(z)-\lambda(t)}, \qquad g_0(z) = z
\]
where the additional degree of freedom $\lambda \in C^{0}\left([0,\infty);\mathbb R\right)$ is called driving function. Historically, the Loewner equation was first introduced in~\cite{lowner1923untersuchungen} in the context of making partial progress on the Bieberbach conjecture~\cite{bieberbach1916uber}, now known as De Branges's theorem~\cite{de1985proof}. In the celebrated work~\cite{schramm2000scaling} it was understood that the Loewner equation encodes a one-to-one correspondence between a family of growing sets $(K_t)_{t \geq 0}$ in the upper half-plane and a real-valued driving function $\lambda$. As an introductory example, each vertically growing curve in the upper half-plane emanating from a real $x \in \mathbb R$ can be generated by choosing the constant driving function $\lambda \equiv x$. Understanding which families of growing sets $(K_t)_{t \geq 0}$ can be generated by choosing an appropriate driving function $\lambda$ has been intensively studied in the context of mathematical physics and geometric function theory. \newline

In the context of mathematical physics, the Loewner equation was used to construct a one parameter family of random fractal curves in the plane, called Schramm-Loewner evolutions, parametrized by $\kappa \in (0,\infty)$. Such curves and their variants have been proven to describe the scaling limits of various discrete statistical physics models, such as the uniform spanning tree Peano curve~\cite{lawler2004conformal}, Gaussian free field level lines~\cite{schramm2013contour}, critical percolation interfaces \cite{smirnov2001critical,camia2006two}, and critical Ising model interfaces~\cite{smirnov2010conformal,schramm2013contour,chelkak2014convergence}.

In the context of geometric function theory, connecting the geometry of growing sets $(K_t)_{t \geq 0}$ with regularity of the driver $\lambda$ has been of great interest. It is a well established result \cite{marshall2005loewner,rohde2018loewner} that each driving function satisfying $\|\lambda\|_{1/2} < 4$, where $\|\cdot\|_{1/2}$ denotes the Hölder-$1/2$ norm, generates a quasiarc emanating from $\lambda(0)$ and possibly meeting the real line non-tangentially. Furthermore, it is known \cite{lind2005sharp,lind2010collisions} that this bound is sharp, i.e. there exist driving functions with norm bigger than $4$, that generate non-locally growing sets such as infinite spirals.

During the last 20 years, a variety of generalizations of the Loewner equation have been studied both in the deterministic and the probabilistic setup such as \cite{rushkin2006stochastic,sheffield2009exploration,pete2016conformally,peltola2024loewner}. A particular generalization originating from unpublished work~\cite{RS} considers complex-valued drivers $\lambda \in C^{0}\left([0,\infty);\mathbb C\right)$. This new variant of the Loewner equation has recently grown in interest and was investigated in both the deterministic \cite{tran2017loewner,lind2022phase} and the probabilistic case \cite{gwynne2023loewner}. In~\cite{tran2017loewner}, it was shown that complex-valued driving functions with sufficiently small Hölder-$1/2$ norm generate simple curves that evolve simultaneously from two ends. We refer to such curves as two-sided, and they arise naturally in the study of complex-valued driving functions.
The main problem of complex-driven Loewner evolutions generalization leads to our motivation as follows. \newline

Complex-driven Loewner evolutions enjoy a larger class of symmetries and admit very different and surprising properties that deviate largely from the usual real-driven chordal setup. 
It remains an open question which subsets of the complex plane admit a natural encoding through complex-valued driving functions, reminiscent of the use of real-valued drivers in the chordal stetting. Motivated by this issue, we make the following main contributions.
\begin{itemize}
    \item\textbf{Criterion for $C^1$-drivers to create simple two-sided curves.} We begin by generalizing the following classical result, which can be found, for example, in~\cite[Remark~4.36]{lawler2008conformally}.

    Each $\lambda \in C^1([0,\infty);\mathbb R)$, restricted to a finite time interval, generates a simple curve emanating from $\lambda(0)$. As already observed in
    \cite[Theorem~3]{lind2022phase}, this is not the case for a complex-valued driving function and its associated two-sided curve. Motivated by this observations, we are interested in understanding the link between $C^1$-drives and two-sided curves. In Section~\ref{Section:3}, especially Proposition~\ref{lem:simple_sufficient}, we prove a sufficient condition for $\lambda \in C^1([0,\infty);\mathbb C)$ to create a simple two-sided curve.
    \item \textbf{$C^{0}$-drivers and two-sided curves.} Recall that in the real-driven chordal case, one usually encodes a growing family of subsets of the upper half-plane with simply connected complements using the following definition.

    A Loewner chain $(g_t)_{t \geq 0}$ is generated by a curve, if there exists a curve $\gamma$ such that $\mathbb H \setminus K_t = C_{\infty}(\mathbb H \setminus \gamma_t)$ for each $t \geq 0$, here $C_{\infty}$ denotes the unbounded connected component. Observe that this definition includes a geometric object (the curve) and a topological property (unbounded connected component of the complement). In the complex-driven case the geometric object is given by a two-sided curve. However, the topological property remains unclear due to multiple reasons and exotic behaviors of specific hulls. We refer to~\cite[Remark~1.10]{gwynne2023loewner} for a discussion on this matter for complex-driven Schramm-Loewner evolutions. To get a better understanding about this issue, in Section~\ref{Section:3.2} we are interested in $\lambda \in C^0\left([0,\infty);\mathbb R\right)$ that create a two-sided curve and prove the following.\newline

    \begin{enumerate}
        \item The condition for $C^1$-drivers in Proposition~\ref{lem:simple_sufficient} to create a simple curve can be seen as the complex-driven analogue of the usual necessary and sufficient condition that $g_t(\gamma(t+s)) \notin \mathbb R$ for each $t \geq 0$ and $s>0$ for the curve to be simple.  In Example~\ref{ex:counterexample} in Section~\ref{Section:3.1}, we prove that although the condition of Proposition~\ref{lem:simple_sufficient} is reminiscent of the real-driven chordal case, it is not a necessary condition. In Example~\ref{ex:counterexample}, we provide a counterexample $\lambda \in C^0\left([0,2];\mathbb C\right)$ that is smooth except at a single point such that the condition of Proposition~\ref{lem:simple_sufficient} is violated, but $\lambda$ still creates a simple two-sided curve.
        \item Interestingly the driver $\lambda \in C^0\left([0,2];\mathbb C\right)$ in Example~\ref{ex:counterexample} provides yet another new and exotic created hull that has been previously not been observed in the complex-driven case. For each $0 \leq t \leq 2$, the driver creates a two-sided curve that stops growing on one end after time $1$ although the whole curve remains simple until time $2$.
        \item Given the two previous points, conditionally on a continuous driver to create a two-sided curve, we discuss an alternative necessary and sufficient condition for the curve to be simple in Lemma~\ref{lem:characterization_simple} in Section~\ref{Section:3.2}. 
    \end{enumerate}
    \item \textbf{Phase transition of the complex linear driver.} In the spirit of~\cite{lind2022phase} and \cite{kager2004exact}, we are interested how perturbations of a fixed driver $\lambda$ influence the geometry of created hulls. For each $c \in \mathbb C$, we define the complex linear driver $\lambda_c$ by setting $\lambda_c(t):=ct$. In Section~\ref{Section:4}, we prove that each $\lambda_c$ creates a two-sided curve, however the geometry of the curve admits a phase transition in terms of $c$. In Theorem~\ref{thm:classification_linear_complete}, we are able to provide a full classification of the complex plane into three disjoint sets. Each complex linear driver $\lambda_c$ creates either a simple, simple with one-end spiraling or another new exotic variant. In the latter, both parts of the two-sided curve only intersect at the origin. One part remains simple for arbitrary large times, the other part forms a Jordan curve rooted in the origin in $c$-dependent finite time and then totally stops growing afterwards. Moreover, the left hull associated to the two-sided curve never swallows the interior component of the complement of the Jordan curve. In other words, the exotic variant disconnects an open set from infinity for arbitrary large times.
    \item \textbf{Optimal Hölder-1/2 constant.} Similar to the real-driven chordal case, it was established in~[Theorem~1.1]\cite{tran2017loewner} that there exists some constant $\sigma \in (1/3,4)$ such that all drivers satisfying $\|\lambda\|_{1/2} < \sigma$ create hulls that are given by two-sided quasiarcs. This bound was recently tightened in [Theorem~3]\cite{lind2022phase} to $\sigma \in (1/3,3.772)$ and using our classification result of the complex linear driver, we are able to provide a sharper upper bound given by $\sigma \in (1/3,3.552)$, which we believe to be non-optimal.
\end{itemize}
Lastly, we refer to \cite[Chapter~4,Chapter~4]{lawler2008conformally,kemppainen2017schramm} for a general introduction on the real-driven chordal Loewner equations and we discuss additional questions in Section~\ref{Sec:Section_6}.

\subsection{Acknowledgments}

L.B.~is supported by the European Research Council (ERC) under the European Union's Horizon 2020 research and innovation programme (101042460): ERC Starting grant ``Interplay of structures in conformal and universal random geometry" (ISCoURaGe) and by the Academy of Finland Centre of Excellence Programme grant number 346315 ``Finnish centre of excellence in Randomness and STructures (FiRST)". Part of this work was carried out during a visit to the Hausdorff Research Institute for Mathematics, funded by the Deutsche Forschungsgemeinschaft (DFG, German Research Foundation) under Germany's Excellence Strategy – EXC-2047/1 – 390685813. At an early stage, part of this work was carried out during the AScI program provided by Aalto University.

I am grateful to many individuals for their support throughout this project. First and foremost, I would like to sincerely thank Fredrik Viklund for accepting me as a master's thesis student, despite my being based in Switzerland at the time. Moreover, he introduced me to the theory of complex driven Loewner chains and provided me a generous stay at KTH at the beginning of this project. I am also grateful to Joan Lind for many insightful discussions and for creating the simulations presented in Theorem~\ref{thm:classification_linear_complete}. My thanks also go to Andrew Gannon for multiple fun and valuable discussions, one of which inspired Example~\ref{ex:counterexample}. I also wish to thank Steffen Rohde for his thoughtful questions, one of which inspired Question~1 in the open question section, and his intuitive insights during my master's thesis defense. Lastly, I want to thank my supervisor Eveliina Peltola for her support, numerous discussions, and comments on earlier drafts of this work.

\section{Preliminaries of complex-driven Loewner chains}\label{Section:2}
In this section, we establish our notation and review fundamental results on complex-driven Loewner chains. Section \ref{Section:2.1} introduces the main definition along with basic properties of these chains. Section \ref{Section:2.2} reviews known results on complex-driven Loewner chains in the setting of hulls, establishes basic properties, and highlights similarities and differences with the usual setup of real-valued driving functions.

\subsection{Basic definitions}\label{Section:2.1}
We begin by defining the main objects of this work.
\begin{defn}[Loewner Differential equation]\label{Def:complex_loewner_chains}
    Let $\lambda \in C^{0}\left([0,\infty),\C\right)$ be a continuous function. The \emph{Loewner differential equation} is defined as the ordinary differential equation with initial value problem given by 
    \begin{equation}\label{eq:Loewner_eq}
        \partial_t g_t(z) = \frac{2}{g_t(z)-\lambda(t)}, \qquad g_0(z) = z,
    \end{equation}
    where $t \geq 0$ and $z \in \mathbb C$ is fixed. We usually call $\lambda$ a driving function, or simply driver.
\end{defn}
We make the following remarks.
\begin{itemize}
    \item For a fixed $z \in \C$, we denote its \emph{blow-up time} as follows
    \[
    T_z := \inf \{t \geq 0 \colon \lim_{s \uparrow t} |g_s(z)-\lambda(s)| = 0\}.
    \]
    \item For fixed $t \geq 0$, we denote the \emph{left hull at time $t$} by
    \[
    L_{t}(\lambda) := \{z \in \mathbb C \, \colon \, T_z \leq t\},
    \]
    where we usually just write $L_t$ instead of $L_t(\lambda)$. In other words, the left hull at time $t$ is the set of all the points that experienced a blow-up until time $t$. 
    \item For any $z \in \C \setminus \{\lambda(0)\}$, the solution to~(\ref{eq:Loewner_eq}) is unique and exists up to a strictly positive time $t \in (0,T_z)$. Consider the set $\{(t,w) \in [0,\infty) \times \C \, \colon \, |w-\lambda(t)| \geq \varepsilon\}$, where $\varepsilon>0$. On this set, define the function
    \[
    (t,w) \mapsto F(t,w):= \frac{2}{w-\lambda(t)}.
    \] 
    Since $\lambda$ is continuous by Definition~\ref{Def:complex_loewner_chains}, the function $F(\cdot,w)$ is continuous. For each fixed $t \in [0,\infty)$, a simple algebraic manipulation yields
    \[
    \big|F(t,w_1)-F(t,w_2)\big| \leq \frac{2}{\varepsilon} \big|w_1-w_2\big|,
    \]
    and thus $w \mapsto F(t,w)$ is Lipschitz continuous on $\{(t,w) \in [0,\infty) \times \C \colon |w-\lambda(t)| \geq \varepsilon\}$. Using the basic theory of ordinary differential equations, we obtain the existence of a unique solution of~(\ref{eq:Loewner_eq}) for each $z \in \mathbb C \setminus \{\lambda(0)\}$ up to some strictly positive $t \in (0,T_z)$. Furthermore, observe that $L_0=\{\lambda(0)\}$.
    \item Since each fixed $z \in \C$ admits a unique solution, we collect all of them together. We call a collection of maps $(g_t)_{t \geq 0}$ 
    \emph{(forward) complex Loewner chain driven by $\lambda$}, if for each $z \in \C$, the function $(g_t(z))_{0 \leq t < T_z}$ solves the Loewner differential equation given by~(\ref{eq:Loewner_eq}).
    \item We are also interested in a centered version of Definition~\ref{Def:complex_loewner_chains}. We denote by $f_t(\cdot) = g_t(\cdot) - \lambda(t)$ the \emph{(forward centralized) complex Loewner chain driven by $\lambda$} and usually just write centered Loewner chain.
    \item In the context of Loewner theory, it is also important to have a backwards running version of Definition~\ref{Def:complex_loewner_chains}. For fixed $z \in \C$, the \emph{backward Loewner equation} is defined as the ODE with initial value problem given by 
    \begin{equation}\label{eq:backwards_loewner_eq}
        \partial_t h_t(z) = \frac{-2}{h_t(z)-\lambda(t)}, \qquad h_0(z) = z,
    \end{equation}
    where $t \geq 0$.
    \item We want to emphasize main differences and similarities with respect to the original version of the Loewner equation. To differentiate between our setup and the already well studied case, we will refer to the original setup as \emph{the real-driven chordal case}. If we speak about the real-driven chordal case, we restrict Definition~\ref{Def:complex_loewner_chains} by only considering $z \in \overline{\mathbb H}$ and real-valued drivers $\lambda \in C^{0}\left([0,\infty) ; \mathbb R\right)$. Following the same remarks as above, for each $t \geq 0$, instead of writing the left hull $L_t$ as the set of all points with $T_z \leq t$, we are going to write $K_t := \{z \in \overline{\mathbb H} \, \colon \, T_z \leq t\}$ as an indication of working in the chordal case.
\end{itemize}

\subsection{Properties of left and right hulls}\label{Section:2.2}

For a brief moment, let us review some facts about the real-driven chordal case. We refer to~\cite[Chapter~4]{lawler2008conformally} as our main reference. It can be shown that the forward Loewner chain $(g_t)_{t \geq 0}$ compromises a family of maps such that for each $t \geq 0$, the map $g_t \colon \mathbb H \setminus K_{t} \to \mathbb H$ is the unique conformal map which satisfies the normalization at infinity given by
\begin{equation}
    g_t(z) = z + \frac{2t}{z}+O(|z|^{-2}) \qquad \text{ as } z \to \infty.
\end{equation}
A natural question is how this setup generalizes to complex-driven Loewner chains. \newline

\textbf{The left/right hull.} To justify that $g_t \colon \C \setminus L_{t} \to g_t(\mathbb C\setminus L_t)$ is indeed the unique conformal map similar to the real-driven chordal case, recall that all the proofs in~\cite[Chapter~4]{lawler2008conformally} only use the structure of the underlying Loewner differential equation given by Definition~\ref{Def:complex_loewner_chains} and its remarks. Analogously to the real-driven chordal case, one might expect a fixed reference domain for the mapping-out function, for instance
\[ 
g_t \colon \mathbb C \setminus L_{t} \to \C,
\]
however this is impossible for multiple reasons. Let us provide an elementary reasoning in the spirit of Definition~\ref{Def:complex_loewner_chains}. For each $t \geq 0$, it is impossible that $\lambda(t)$ lies in the image of $g_t$. In fact, if there is some $t \geq 0$, such that $g_t(z)=\lambda(t)$, then this implies $z \in L_{t}$, which is impossible by the definition of $T_z$. As a result, for each $t \geq 0$, the image of $g_t$ is always a proper subset of the complex plane, while the complement of its image contains at least $\lambda(t)$. We define this "residual" set as $R_{t}(\lambda) := \C \setminus g_t(\C \setminus L_{t}(\lambda))$, which we are going to call \emph{the right hull at time $t$}. We usually write $R_t$ instead of $R_t(\lambda)$. \newline

\textbf{The expansion around infinity.} In the real-driven chordal case, for fixed $t \geq 0$, the map $g_t(z)$ behaves like the identity map around infinity. Precisely speaking, $g_t$ admits the Laurent expansion around infinity given by
\begin{equation}\label{eq:normalization_infinity}
    g_t(z) = z + \frac{2t}{z} + O(|z|^{-2}), \qquad \text{ as } z \to \infty.
\end{equation}
We are going to provide a quick heuristic argument that the same expansion is valid in the complex-driven case. Fix some $t \geq 0$, by similar means as presented in~\cite[Chapter~4.1]{lawler2008conformally}, we have 
\[
g_t(z) = z + o(1), \qquad \text{ as } z \to \infty.
\]
As a result, the map $g_t$ admits the following expansion around infinity
\[
g_t(z) = z + \frac{a(t)}{z} + \frac{b(t)}{z^2} + O(|z|^{-3}), \qquad \text{ as } z \to \infty,
\]
where $a(t),b(t) \in \mathbb C$ are coefficients, yet to be determined. Using this expansion, for large enough $z$, the Loewner differential equation reads
\[
\partial_t g_t(z) = \frac{2}{g_t(z)-\lambda(t)} = \frac{2}{z\left(1+ \frac{a(t)}{z^2} - \frac{\lambda(t)}{z} + \dots\right)} = \frac{2}{z} \left(1-\frac{a(t)}{z^2}+\frac{\lambda(t)}{z}+\dots\right).
\]
Rewriting the above, we have
\[
\partial_t g_t(z) = \frac{2}{g_t(z)-\lambda(t)} = \frac{2}{z} + \frac{2\lambda(t)}{z^2} + O(|z|^{-3}), \qquad \text{ as } z \to \infty.
\]
Comparing coefficients, we have $a(t)=2t$ and $b(t)=2\int^{t}_0 \lambda(s) ds$, so
\[
g_t (z) = z + \frac{2t}{z} + \frac{2\int^{t}_0 \lambda(s) ds}{z^2} + O (|z|^{-3}), \qquad \text{ as } z \to \infty.
\]
As a result, for fixed $t \geq 0$, we have
\[
g_t(z) = z + \frac{2t}{z} + O(|z|^{-2}),
\]
where only the $O$-term depends on the complex-valued driver $\lambda(s)$ for $0 \leq s \leq t$ but the first order term remains real-valued. The following definition summarizes the above discussion.
\begin{defn}[Left/Right hulls]\label{Def:Left/Right_hulls}
    Let $\lambda \in C^{0}\left([0,\infty),\mathbb C\right)$ be a continuous function. We define the left hull as $L_{t}:= \{z \in \C \, \colon \,T_z \leq t\}$, i.e. the complement of the domain of $g_t$. We define the right hull as $R_{t} := \C \setminus g_t(\C \setminus L_t)$, i.e. the complement of the co-domain of $g_t$. All in all, 
    \[
    g_t \colon \C \setminus L_t \to \C \setminus R_t,
    \]
    is the unique conformal map that satisfies the normalization given by~(\ref{eq:normalization_infinity}).
\end{defn}
Let us make the following remarks.
\begin{itemize}
    \item By definition, it is clear that $L_{t}$ is closed and bounded, and thus compact. Furthermore, it was shown in~\cite[Lemma~2.4]{gwynne2023loewner} that $L_{t}$ is also connected.
    \item The left hull is a monotonically increasing (not necessarily strictly increasing) set in time, i.e. $L_{s} \subseteq L_{t}$ for each $0 \leq s \leq t$. It is natural to ask, whether or not a similar statement for the right hull is also true. This is not the case, and we refer to~\cite[Remark~1.11]{gwynne2023loewner} or Lemma~\ref{lem:5} for a quick argument.
    \item On the deterministic side, it has been shown in~\cite[Theorem~3]{lind2022phase} that it is possible for the complements of $L_{t}$ and $R_{t}$ to be at least two-connected. In the work~\cite{gwynne2023loewner} one considers the left and right hulls, driven by $\lambda(t) = \sqrt{a}B_t + i \sqrt{b} \widetilde{B}_t$, where $B,\widetilde{B}$ are (standard) real Brownian motion with $\text{Cov}(B_t,\widetilde{B}_t)=\left(ct/\sqrt{ab}\right)$ for $c \in [-\sqrt{ab},\sqrt{ab}]$. Similarly to the deterministic case, in~\cite[Theorem~1.10]{gwynne2023loewner} and its remarks the following two statements are proven. If $a,b \neq 0$, then for any $c \in [-\sqrt{ab},\sqrt{ab}]$ almost surely for each $t>0$ the complement of the left hull is at least two-connected. Furthermore, the complementary component that is disconnected from infinity has at least non-empty interior. Dealing with left (resp. right hulls) with multiply connected complements is the major key difference compared to the real-driven chordal driven case.
    \item It can be shown that for each $t \geq 0$, the map $g_t$ can be extended continuously to all the points that just experienced a blow-up. Precisely, in~\cite[Lemma~2.5]{gwynne2023loewner} it was proven that for each $z \in \C$ with $T_z \leq C < \infty$, one has 
    \[
    \lim_{t \uparrow T_z} |g_t(z)-\lambda(t)| = 0.
    \]
    As a result, the continuous extension of $g_t$ maps all points in $L_t \setminus \cup_{0 \leq s <t} L_s$ to $\lambda(t)$, i.e. the driver at the time $t$.
    \item Using Definition~\ref{Def:complex_loewner_chains} it is a classical fact that one can bound the real (resp. imaginary) part of the left (resp. right hull). In~\cite[Lemma~12]{lind2022phase} it was proven that for each $T \geq 0$ the left (resp. right) hulls are contained in the real (resp. imaginary) strips given by
    \begin{align*}
        L_T &\subseteq \{z \in \mathbb C \, \colon \, \min_{0 \leq t \leq T}\re(\lambda(s)) \leq \re(z) \leq \max_{0 \leq t \leq T}\re(\lambda(t)) \,\},\\
        R_T &\subseteq \{z \in \mathbb C \, \colon \, \min_{0 \leq t \leq T}\im(\lambda(s)) \leq \im(z) \leq \max_{0 \leq t \leq T}\im(\lambda(t))\, \}.
    \end{align*}
    \item Denote by $\|\cdot\|_{1/2}$ the Hölder-1/2 norm. In~\cite[Theorem~1.1]{tran2017loewner} it was shown that there exists $\sigma \in(1/3,4)$ such that if $\|\lambda \vert_{[0,T]}\|_{1/2} \leq \sigma$, then the left hull up to time $T$ is given by a quasiarcs. Considering the one-parameter family of driving functions given by $(c\sqrt{1-t})_{c \in \mathbb C}$. It was shown in~\cite[Theorem~3]{lind2022phase} that $\sigma \leq 3.722$ and we are going to give a sharper upper bound for $\sigma$ in Lemma~\ref{lem:optimal_sigma}, which we believe to be non-optimal.
\end{itemize}

\textbf{Properties of left/right hulls.}

As already observed and mentioned in~\cite[Chapter~4,Chapter~2,Chapter~2]{tran2017loewner,gwynne2023loewner,lind2022phase} complex-driven Loewner chains conserve the translation, scaling property of the real-driven chordal case while also admitting three additional reflection principles and a duality principle. Roughly speaking, the duality principle states that the left (resp. right) hull is a rotated right (resp. left) hull created by a rotated and time-reversed version of the original driver.

For the deterministic setup, the following properties are already stated in~\cite[Chapter~4,Chapter~2]{tran2017loewner,lind2022phase}. For the sake of completeness, we are going to provide the proofs in the deterministic setting and the experienced reader might skip to the additional reflection property stated in Lemma~\ref{lem:reflections} or to Section~\ref{Section:3}. Let us start with the translation and scaling property.

\begin{lem}[Translation and scaling properties]\label{Lem:4}
    The left (resp. right) hulls driven by $\lambda$ satisfy the following translation and scaling property.
    \begin{itemize}[label={}]
        \item \textbf{Translation Property:} For fixed $a \in \C$ and arbitrary $t \geq 0$, the left (resp. right) hulls satisfy
            \begin{align*}
                L_{t}(\lambda+a) &= a + L_{t}(\lambda)\tag{\textcolor{blue}{TL}}\label{trans:0},\\
                R_{t}(\lambda+a) &= a + R_{t}(\lambda)\tag{\textcolor{blue}{TR}}\label{trans:1}.
            \end{align*}
            
        \item \textbf{Scaling Property:} For fixed $a > 0$ and arbitrary $t \geq 0$, the left (resp. right) hulls satisfy
            \begin{align}
                L_{t}(a\lambda(\cdot/a^2)) &= a L_{t/a^2}(\lambda)\tag{\textcolor{blue}{SL}}\label{scaling:0},\\
                R_{t}(a\lambda(\cdot/a^2)) &= a R_{t/a^2}(\lambda)\tag{\textcolor{blue}{SR}}\label{scaling:1},
            \end{align}
            here $\lambda(\cdot/a^2)$ denotes the mapping $s \mapsto \lambda(s/a^2)$.
    \end{itemize}
\end{lem}
Before beginning the proof, we briefly outline the strategy. We will prove each statement only for the left hull; once the duality principles in Lemma~\ref{lem:6} are established, the corresponding results for the right hull follow automatically.
\begin{proof}[Proof of Lemma~\ref{Lem:4}]
We first show~(\ref{trans:0}). Fix $a \in \C$, it suffices to prove that $z \in \mathbb C \setminus L_{t}(\lambda+a)$ if and only if $z \in \mathbb C \setminus \left(a + L_{t}\right)$. Fix $z \in \mathbb C \setminus \left(a + L_{t}\right)$ and denote by $\widetilde{g}_s(z) := g_s(z-a) + a$, where $s \in [0,t]$ and $(g_s)_{0 \leq s \leq t}$ is the Loewner chain driven by $\lambda$. Now for each $0 \leq s \leq t$ by direct computation,
\[
\partial_s \widetilde{g}_s(z) = \frac{2}{g_s(z-a) - \lambda(s)} = \frac{2}{\widetilde{g}_s(z) - (\lambda(s)+a)}, \qquad \widetilde{g}_0(z)=z.
\]
Consequently, $(\widetilde{g}_s(z))_{0\leq s \leq t}$ is a solution to the Loewner equation driven by $\lambda+a$. By uniqueness of the solution of the Loewner equation for fixed $z$, we conclude that $z \in \mathbb C \setminus L_t(\lambda+a)$. The other direction follows by a similar argument. Fix $z \in \mathbb C\setminus L_t(\lambda+a)$, and consider $\widetilde{g}_s(z):=g_s(z+a)-a$, where $(g_s)_{0 \leq s \leq t}$ is the Loewner chain driven by $\lambda+a$. Now, since $z \in \mathbb C \setminus L_t(\lambda+a)$, for each $0 \leq s \leq t$ by direct computation,
\[
\partial_s \widetilde{g}_s(z-a) = \frac{2}{g_s(z)-\left(\lambda(s)+a\right)} = \frac{2}{\widetilde{g}_s(z-a)-\lambda(s)}, \qquad \widetilde{g}_0(z-a)=z-a.
\]
Consequently, $(\widetilde{g}_s(z-a))_{0 \leq s \leq t}$ is a solution to the Loewner equation driven by $\lambda$. By uniqueness of the solution, we conclude that $z-a \in \mathbb C \setminus L_t(\lambda)$, so $z \in \mathbb C \setminus L_t(\lambda)+a$. This finishes the proof of~(\ref{trans:0}) and it is left to prove~(\ref{scaling:0}), where we apply a contraposition argument. Consider any $z \in \mathbb C \setminus (a L_{t/a^2}(\lambda))$ and we want to show that $z \in \mathbb C \setminus \left(L_{t}(a\lambda(\cdot/a^2))\right)$. Denote by $\widetilde{g}_s(z) := a g_{s/a^2}(z/a)$, where $s \in [0,t]$ and $(g_s)_{0 \leq s \leq t/a^2}$ is the Loewner chain driven by $\lambda$. Then since $z/a \notin L_{t/a^2}(\lambda)$, for each $0 \leq s \leq t$ by direct computation
\[
\partial_{s} \widetilde{g}_s(z) = \partial_s a{g}_{s/a^2}(z/a) = \frac{2}{a(g_{s/a^2}(z/a)-\lambda(s/a^2))} = \frac{2}{\widetilde{g}_s(z) - a\lambda(s/a^2)}, \qquad \widetilde{g}_0(z) = z.
\]
Since the solution of the Loewner equation is unique for fixed $z$, we conclude that $z \in \mathbb C \setminus \left(L_{t}(a\lambda(\cdot/a^2))\right)$. The other inclusion follows a similar argument, which establishes~(\ref{scaling:0}). This finishes the proof.
\end{proof}

Let $A \subseteq \C$ and denote by $\ast$ the complex conjugation, then we denote by $A^{\ast}:=\{-z^{\ast} \mid z \in A\}$ the complex conjugation of $A$. In the complex-driven case, we obtain the following three additional reflection properties.
\begin{lem}[Reflection properties]\label{lem:reflections}
For each $t \geq 0$, the left hull $L_t$ and the right hull $R_t$ satisfy the following reflection principles.
    \begin{align}
                L_{t}(\lambda^{\ast}) &= L^{\ast}_{t}(\lambda)\tag{\textcolor{blue}{RR}}\label{reflection:Real},\\
                L_{t}(-\lambda^{\ast}) &= -L^{\ast}_{t}(\lambda)\tag{\textcolor{blue}{RI}}\label{Reflection:Imaginary}, \\
                L_{t}(-\lambda) &= -L_{t}(\lambda)\tag{\textcolor{blue}{RO}}\label{Reflection:Origin},
    \end{align}
where the same statements are true for the right hulls $(R_{t})_{t \geq 0}$.
\end{lem}
\begin{proof}
    We start with~(\ref{reflection:Real}). Take, any $z \in \mathbb C \setminus L^{\ast}_{t}(\lambda)$, we first show that $z \in \mathbb C \setminus L_{t}(\lambda^{\ast})$. Denote by $\widetilde{g}_s(z):= g^{\ast}_s(z^{\ast})$, where $(g_s)_{0 \leq s \leq t}$ is the Loewner chain driven by $\lambda$. Since $z^{\ast} \in \mathbb C \setminus L_t$, by direct computation
    \[
    \partial_s \widetilde{g}_s(z) = \left(\frac{2}{g_s(z^{\ast})-\lambda(s)}\right)^{\ast} = \frac{2}{\widetilde{g}_s(z)-\lambda^{\ast}(s)}, \qquad \widetilde{g}_0(z)=z,
    \]
    for any $0 \leq s \leq t$. By uniqueness of the solution of the Loewner equation for fixed $z$, we conclude $z \in \mathbb C \setminus L_{t}(\lambda^{\ast})$. The other implication follows analogously, and this finishes~(\ref{reflection:Real}). Next, we want to prove statement~(\ref{Reflection:Imaginary}). Take any $z \in \mathbb C \setminus \left(-L^{\ast}_t\right)$, and denote by $\widetilde{g}_s(z):=-g^{\ast}_s(-z^{\ast})$, where $(g_s)_{0 \leq s \leq t}$ is the Loewner chain driven by $\lambda$. Since $z \in \mathbb C \setminus \left(-L^{\ast}_t\right)$, by direct computation
    \[
     \partial_s \widetilde{g}_s(z) = \left(\frac{-2}{g_s(-z^{\ast})-\lambda(s)}\right)^{\ast} = \frac{2}{\widetilde{g}_s(z)-(-\lambda^{\ast}(s))}, \qquad \widetilde{g}_0(z)=z,
    \]
    where $0 \leq s \leq t$. As a result $z \in \mathbb C \setminus L_t(-\lambda^{\ast})$ and the other inclusion follows by considering a similar transformation. It remains to prove~(\ref{Reflection:Origin}), but we can conclude this directly by combining the two previous properties. For each $t \geq 0$, applying first~(\ref{Reflection:Imaginary}) and then~(\ref{reflection:Real}) gives
    \[
    L_t(-\lambda) = -L^{\ast}_t(\lambda^{\ast}) = -L_t(\lambda).
    \]
    This finishes the proof for the reflection properties of the left hulls.
\end{proof}
Next we prove the concatenation properties in the complex-driven case.
\begin{lem}[Concatenation properties]\label{lem:5}
Fix $t \geq 0$ and for each $s \geq 0$ we denote by $L_{t,t+s}:=L_s(\lambda(t+\cdot))$ the left hull driven by $\lambda(t+\cdot)$ up to time $s$. Similarly, we denote by $R_{t,t+s}:=R_s(\lambda(t+\cdot))$ the corresponding right hull up to time $s$. We also denote by $(g_{t,t+s})_{s \geq 0}$ the Loewner chain driven by $\lambda(t+\cdot)$. For each $s,t \geq 0$, the following statements are true.
\begin{align}
    g_{t+s} &= g_{t,t+s} \circ g_t \vert_{\C \setminus L_{t+s}},\tag{\textcolor{blue}{C}}\label{concat:0} \\
    g_t(L_{t+s}\setminus L_{t}) &= L_{t,t+s} \setminus R_{t}\tag{\textcolor{blue}{CST}}\label{concat:1} \\
    L_{t+s} &= L_{t} \cup g^{-1}_t(L_{t,t+s} \setminus R_{t})\tag{\textcolor{blue}{CL}}\label{concat:2} \\
    R_{t+s} &= R_{t,t+s} \cup g_{t,t+s}(R_t \setminus L_{t,t+s})\tag{\textcolor{blue}{CR}}\label{concat:3}.
\end{align}
See Figure~\ref{figure:3} for an illustration.
\end{lem}
\begin{figure}[ht]
\centering
\includegraphics[width=0.8\textwidth]{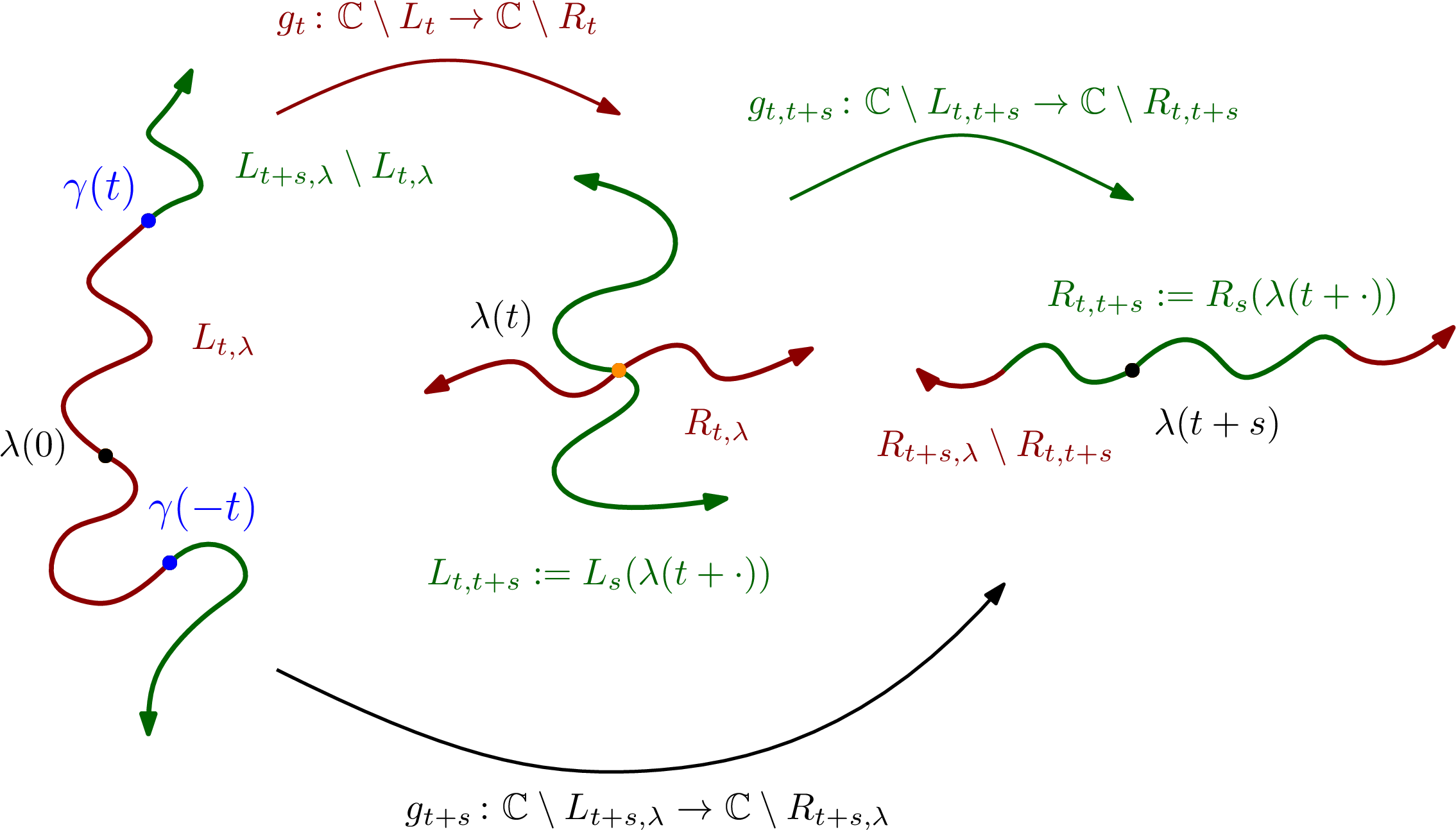}
\caption{Concatenation principle (Lemma~\ref{Lem:4}) for a complex-driven Loewner chain in the case where both the left as well as the right hulls are given by a two-sided simple curve. Observe that the tips of the curve $\gamma(t)$ and $\gamma(-t)$ (\textcolor{blue}{blue}), are mapped to $\lambda(t)$ (\textcolor{orange}{orange}) by Lemma~\ref{lem:regularity_mappingout}. Moreover, applying the mapping out function $g_{t,t+s}$ pushes the already grown red hull $R_{t}$ further out, breaking any kind of set-monotonicity for the right hull, while $R_{t,t+s}$ starts growing around $\lambda(t+s)$.}
\label{figure:3}
\end{figure}
\begin{proof}[Proof of Lemma~\ref{lem:5}]
Fix $t \geq 0$, we start with statement~(\ref{concat:0}) and consider $(\widetilde{g}_s)_{s \geq 0}$, where $\widetilde{g}_s:= g_{t,t+s} \circ g_t$. For any $z \in \mathbb C \setminus L_{t+s}$ by direct computation
\[
\partial_s \widetilde{g}_s(z) = \frac{2}{g_{t,t+s}(g_t(z)) - \lambda(t+s)} = \frac{2}{\widetilde{g}_s(z)-\lambda(t+s)}, \qquad \widetilde{g}_0(z) = g_t(z).
\]
Using the uniqueness of Loewner chains for fixed $z$, $(\widetilde{g}_s)_{s \geq 0}$ is the Loewner chain driven by $\lambda(t+\cdot)$ and so statement~(\ref{concat:0}) follows. For~(\ref{concat:1}) take some $z \in L_{t+s} \setminus L_t$, then $g_t(z)$ is well defined and by definition not in $R_t$. We now show that $g_t(z) \in L_{t,t+s}$ by proving that $t <T^{\lambda(t+\cdot)}_{g_t(z)} \leq t+s$, where the blow-up time is associated to the shifted driver $\lambda(t+\cdot)$. Since $z \in L_{t+s} \setminus L_t$, we have $t<T_z \leq t+s$ so that $(g_u(z))_{0 \leq u < T_z}$ solves the Loewner equation driven by $\lambda$. A direct computation gives
\[
\partial_s g_{t,t+s}(g_t(z)) = \partial_s g_{t+s}(z) = \frac{2}{g_{t+s}(z)-\lambda(t+s)} = \frac{2}{g_{t,t+s}(g_t(z))-\lambda(t+s)}, \qquad g_0(z) = z,
\]
where we used~(\ref{concat:0}) for the first and last equality.
By uniqueness of the solution of the Loewner equation, using $t < T^{\lambda}_z \leq t+s$ implies that $0 \leq T^{\lambda(t+\cdot)}_{g_t(z)} \leq s$ and hence~(\ref{concat:1}) follows. The next statement~(\ref{concat:2}) follows directly by applying the conformal map $g^{-1}_t$ to previous statement (\ref{concat:1}). The last property~(\ref{concat:3}) follows similarly and this finishes the proof.
\end{proof}
We make the following remarks about the concatenation principle.
\begin{itemize}
    \item Observe that both $L_{t,t+s}$ and $R_{t}$ start to grow from $\lambda(t)$. Consequently, for each $t,s \geq 0$ one has $\lambda(t) \in L_{t,t+s} \cap R_{t}$.
    \item After establishing the duality principle in Lemma~\ref{lem:6}, one can conclude that the set $R_{t}$ is also connected for each $t \geq 0$. Recall that a union of two connected sets is also connected if they at least share one common point (see Lemma~\ref{lem:Union_connected} for a reference). As a consequence, using $\lambda(t) \in L_{t,t+s} \cap R_{t}$, we conclude that for each $t,s \geq 0$, the closed set $L_{t,t+s} \cup R_{t}$ is also connected.
\end{itemize}

\textbf{Duality Property.} A natural question is whether the right hulls $(R_t)_{t \geq 0}$ admit an alternative characterization directly in terms of the driving function. In unpublished work~\cite{RS}, Rohde and Schramm observed that the left (resp. right) hulls admit an additional symmetric structure: The left (resp. right) hull is a rotation of the right (resp. left) hull generated by the Loewner equation with a time-reversed, rotated driver. In the probabilistic setting~\cite[Chapter~2.1]{gwynne2023loewner}, using a variant of planar Brownian motion as the driving function, an analogous statement was proved. To our knowledge, no published proof exists in the deterministic setting; we therefore provide one here. We begin by establishing the time-reversal property in the complex-driven case.
\begin{lem}[Time-reversal property]\label{lem:time_reversal}
    Let $\lambda \in C^{0}\left([0,\infty);\C\right)$ be continuous and let $(h_t)_{0 \leq t \leq T}$ be the solution to the backwards running Loewner equation given by 
    \begin{align*}
        \partial_t h_t(w) &= \frac{-2}{h_t(w)-\lambda(T-t)}, \qquad h_0(w) = w, \\
    \end{align*}
    where $w \in \C$ is fixed. Set $z:=h_T(w)$ and define $g_t(z):=h_{T-t}(w)$ for each $0 \leq t \leq T$. Then $(g_t(z))_{0 \leq t \leq T}$ solves the forward Loewner equation given by
    \begin{align*}
        \partial_t g_t(z) &= \frac{2}{g_t(z)-\lambda(t)}, \qquad g_0(z) = z, \\
    \end{align*}
    and for the final time $t=T$, one has $g_T(h_T(w))=w$.
\end{lem}
\begin{proof}
Fix some $w \in \C$ such that $(h_t(w))_{0 \leq t \leq T}$ solves the backwards Loewner equation driven by $\lambda(T-\cdot)$ and consider $(g_t(z))_{0 \leq t \leq T}$ as given in the assumptions. We show that $g_t(z)$ satisfies the (forward) Loewner equation driven by $\lambda$. Indeed, by applying the chain rule one has
\[
\partial_t g_t(z) = \partial_t (h_{T-t}(w))= \frac{2}{h_{T-t}(w)-\lambda(T-(T-t))}  = \frac{2}{g_t(z)-\lambda(t)}, \qquad g_0(z)=h_T(w)=z,
\]
and so the first claim follows. The second claim follows immediately by observing $g_T(h_T(w))=g_T(z)=h_0(w)=w$. This finishes the proof.
\end{proof}
We are now ready to state and prove the duality property. After completing the proof, the translation, scaling, and reflection properties of the right hull follow directly from the corresponding results for the left hull.
\begin{lem}[Duality principle]\label{lem:6}
Let $\lambda \in C^{0}\left([0,\infty);\C\right)$ be continuous. Then for each $t \geq 0$, the corresponding left and right hulls satisfy the following duality principle.
\begin{align}
    L_{t} &= i R_{t}(-i\lambda(t-\cdot))\tag{\textcolor{blue}{DL}}\label{duality:1},\\
    R_{t} &= i L_{t}(-i\lambda(t-\cdot))\tag{\textcolor{blue}{DR}}\label{duality:2}. 
\end{align}
\end{lem}
\begin{proof}
We first prove the first duality principle~(\ref{duality:1}). By the definition of the right hull the following statements are equivalent.
\begin{align}
    z \in \C \setminus iR_{t}(-i\lambda(t-\cdot)) \qquad \iff \qquad& -iz \in g_t\left(\C \setminus L_{t}(-i\lambda(t-\cdot))\right)\\
   \qquad\iff \qquad& -iz = g^{-i\lambda(t-\cdot)}_t(\zeta) \qquad \text{ for some }\, \zeta \in \mathbb{C} \setminus L_{t}(-i\lambda(t-\cdot))\label{eq:duality_proof},
\end{align}
where $g^{-i\lambda(t-\cdot)}_t$ denotes the solution of the Loewner equation driven by $-i\lambda(t-\cdot)$ up to time $t$. Fix a point $z \in \C \setminus L_{t}$ and by above, it suffices to show that $z$ satisfies the right hand side of~(\ref{eq:duality_proof}). The rest of the proof is dedicated to this task and the idea is to apply a time-reversal argument to get the $-i\lambda(t-\cdot)$ driving function while accounting for the $-iz$ term.

We start by establishing the right driving function. Consider any $w \in \C$, which we choose later, such that $-iw \in \C \setminus L_{t}(-\lambda)$. Denote by $h_u(w):=i\hat{g}_u(-iw)$, where $(\hat{g}(-iw))_{0 \leq u \leq t}$ solves the Loewner equation driven by $-\lambda$. Then since $-iw \in \C \setminus L_{t}(-\lambda)$, one has
\[
\partial_u h_u(w) =i\,\partial_u \hat{g}_u(-iw) = \frac{2i}{\hat{g}_u(-iw)+\lambda(u)} = \frac{-2}{h_u(w) + i\lambda(u)},
\]
for any $u \in [0,t]$. Denoting by $\hat{\lambda}(u):=-i\lambda(t-u)$, then the above can be rewritten as follows
\[
\partial_u h_u(w) = \frac{-2}{h_u(w)-\hat{\lambda}(t-u)}, \qquad h_0(w)=i\hat{g}_0(-iw)=w.
\]
Consequently, $(h_u(w))_{0 \leq u \leq t}$ solves the backwards Loewner equation driven by $\hat{\lambda}(t-\cdot)$ up to time $t$ and especially $h_t(w) \in \C \setminus L_{t}(-i\lambda(t-\cdot))$. We are now ready to apply the time reversal argument by choosing a suitable point $w$.

Take $w=-iz$, since $z \in \C \setminus L_{t}$ by the reflection property of the left hull~(\ref{Reflection:Origin}) we have $-iw=-z \in \C \setminus L_{t}(-\lambda)$ and by above $(h_u(w))_{0 \leq u \leq t}$ solves the backwards running Loewner equation driven by $\hat{\lambda}(t-\cdot)$. Take $\zeta:=h_t(w)$ and define $k_u(\zeta):=h_{t-u}(w)$ for each $0 \leq u \leq t$. Then by applying Lemma~\ref{lem:time_reversal} we know that $(k_u(\zeta))_{0 \leq u \leq t}$ solves the forward running Loewner equation given by
\[
\partial_u k_u(\zeta) = \frac{2}{k_u(\zeta)-\hat{\lambda}(u)} = \frac{2}{k_u(\zeta)-\left(-i\lambda(t-u)\right)}, \qquad k_0(\zeta) = \zeta,
\]
and $k_t(h_t(w))=w=-iz$. By the uniqueness of the solution of the Loewner equation for fixed $z$, $(k_t(\zeta))_{0 \leq u \leq t}$ is the unique solution for the forwards running Loewner equation driven by $-i\lambda(t-\cdot)$. Also $\zeta=h_t(w) \in \C \setminus L_{t}(-i\lambda(t-\cdot))$ and since $k_t(\zeta)=-iz$, we conclude that $z$ satisfies the right hand side of~(\ref{eq:duality_proof}). This finishes the proof of the first statement~(\ref{duality:1}). For the second statement~(\ref{duality:2}), observe that for each $t \geq 0$ we have
\[
i L_t(-i\lambda(t-\cdot)) = -iL_t(i\lambda(t-\cdot)) = -i \left(i R_t\left(-i(i\lambda(t-(t-\cdot)\right) \right) = R_t(\lambda),
\]
where we used~(\ref{Reflection:Origin}) and (\ref{duality:1}) in this order. This finishes the proof.
\end{proof}

\textbf{Upshot.} We advice the reader to consult Figure~\ref{figure:3} for an illustration. Fix some $0 \leq t \leq t+s \leq T$, using the concatenation principle of left hulls~(\ref{concat:1}) one has
\[
g_t(L_{t+s} \setminus L_{t}) = L_{t,t+s} \setminus R_{t}.
\]
The conformal map $g_t$ maps $L_{t+s} \setminus L_{t}$ on $L_{t,t+s}$ (\textcolor{Green}{green}) while growing the residual set right hull $R_{t}$ (\textcolor{BrickRed}{red}). While $L_{t,t+s}$ grows by running the Loewner equation driven by $\lambda(t+\cdot)$, the right hull grows by running another forward Loewner equation driven by $-i\lambda(t-\cdot)$ and subsequently rotating the hull. Compared to the real-driven chordal case, our mapping-out function $g_t$ does not have any reference domain such as the upper half-plane. Instead, for each $t \geq 0$, the conformal map $g_t$ cuts into the complex plane by growing the connected and closed set $R_t$ and mapping the increment of the left hull to $L_{t,t+s}$, which both start at $\lambda(t)$. \newline
\section{Two-sided curves and complex drivers}\label{Section:3}

In this section, we investigate the relationship between two-sided curves and complex-valued drivers in the deterministic setting. Such connections were previously discussed in~\cite{tran2017loewner,lind2022phase}. To provide context for our definitions and objectives, we first briefly revisit the notion of a driver generating a curve in the real-driven chordal case. We refer the reader to \cite[Chapter~4.4]{lawler2008conformally} for a detailed introduction.

Let $\gamma \, \colon \, [0,\infty) \to \overline{\mathbb H}$ be a curve emanating from the real line and denote by $\gamma_t:=\gamma[0,t]$ its segment up to time $t$. Further denote by $C_{\infty}({\mathbb H} \setminus \gamma_t)$ the unbounded connected component of the complement of the curve segment up to time $t$ then the following definition connects the geometry of chordal hulls $(K_t)_{t \geq 0}$ with a given curve.
\begin{defn}[Driver generates curve]\label{def:generated_by_curve}
    The driving function $\lambda$ of a Loewner chain $(g_t)_{t \geq 0}$ with associated hulls $(K_t)_{t \geq 0}$ generates a curve if there exists a curve $\gamma$ such that $\mathbb H \setminus K_t = C_{\infty}({\mathbb H} \setminus \gamma_t)$ for each $t \geq 0$.
\end{defn}
We emphasize that the above definition combines the existence of a geometric object (a curve) with a topological property (the unbounded connected component). In this section, we are interested in possible generalizations of Definition~\ref{def:generated_by_curve} to the complex-driven case. Let us make the following remarks as guidelines for our work.
\begin{itemize}
    \item It seems difficult to state a generalization of Definition~\ref{def:generated_by_curve} in the case of complex-driven Schramm-Loewner evolutions. We refer to \cite[Remark~1.11]{gwynne2023loewner} for a discussion on this statement. In the probabilistic and deterministic settings, the issue can be summarized as follows.
    \item Reviewing the work of \cite{tran2017loewner}, we later see in this section that for complex-driven Loewner chains the generalization of the geometric object in Definition~\ref{def:generated_by_curve} (the curve) is given by two-sided curve $\gamma : (-\infty,\infty) \to \mathbb{C}$.
     However, identifying the appropriate topological property that connects such two-sided curves with the created left hulls $(L_t)_{t \geq 0}$ proves to be challenging. The difficulty arises from two distinct obstacles.
    \begin{enumerate}
        \item The main obstacle lies in the existence of left hulls generated by two-sided curves with multiply connected complement, where each complementary component satisfies the following property: Each complementary component may be absorbed into the left hull at a strictly later time than the moment it was first disconnected from infinity by the two-sided curve. In other words, a left hull generated by a complex-valued driving function may disconnect open sets from infinity for a positive duration of time.

        For complex Schramm-Loewner evolutions this is almost surely the case, as proven in~\cite[Theorem~1.10]{gwynne2023loewner}. A deterministic example is given in~\cite[Theorem~3]{lind2022phase}. This gives the left hull a "non-locally growing" flavor independent of the position of the tip of the curve, which makes it difficult to relate the complement of left hulls with a specific time segment of the two-sided curve.
        \item Secondly, there might exist a continuous driver which generates $\gamma_s = \gamma_t$ for some $0 \leq s < t \leq T$ where we recall that $\gamma_t$ denotes $\gamma[0,t]$.
    \end{enumerate}
\end{itemize}
Both points make it difficult to fix a single topological property reminiscent of the one in Definition~\ref{def:generated_by_curve} that connects the left hulls $(L_t)_{t \geq 0}$ with the geometric object of a two-sided curve. Motivated by this issue, we are interested in the following questions and prove the following results.
\begin{itemize}
    \item In the real-driven chordal case, it is a fact that each driver $\lambda \in C^1\left([0,T];\mathbb R\right)$ generates a curve in the sense of Definition~\ref{def:generated_by_curve} that is simple. As already observed in~\cite[Theorem~3]{lind2022phase}, this is not true in the complex-driven case. In Section~\ref{Section:3.2}, we prove a sufficient condition on the left and right hulls, to ensure that a driver $\lambda \in C^1\left([0,T];\mathbb C\right)$ creates a left hull given by a simple two-sided curve.
    \item One of the main advantages of working with a driving function that generates a curve in the sense of Definition~\ref{def:generated_by_curve}, is that it is easy to study various properties of the associated hulls such as phase transitions, Hausdorff dimension of the underlying curve and the following equivalent condition to check whether or not the curve is simple.
    \begin{lem}{\cite[Lemma~4.34]{lawler2008conformally}}\label{lem:generated_by_curve}
    Let $\lambda$ be a driving function of a Loewner chain $(g_t)_{t \geq 0}$ that generates a curve $\gamma$. Then the following statements are equivalent.
        \begin{enumerate}
            \item The curve $\gamma$ is simple.
            \item For each $t \geq 0$, one has $g_t\left(\gamma\big(t+s\big) \, \colon \, s \geq 0\right) \cap \mathbb R = \{\lambda(t)\}$.
        \end{enumerate}
    \end{lem}
    Although a “generated by a curve” statement has not yet been established in the complex-driven case we are interested in whether an analogous statement can be obtained in the simplest situation, namely when $\lambda$ generates a left hull $(L_t)_{t \geq 0}$, which is given by a two-sided curve. Recall that by Definition~\ref{Def:Left/Right_hulls}, that the mapping out function $g_t \colon \mathbb C\setminus L_t \to \mathbb C \setminus R_t$ has no reference domain reminiscent of the upper half-plane in the chordal real-driven case. This together with the reflection property~(\ref{reflection:Real}) suggests that any potential generalizations of the second condition above should be of the form $L_{t,t+s} \cap R_t = \{\lambda(t)\}$ for each $0 \leq t \leq t+s < \infty$. While Section~\ref{Section:3.2} indeed establishes that the above condition is sufficient for a driving function $\lambda \in C^{1}\!\left([0,T];\mathbb{C}\right)$ to create a left hull given by a simple two-sided curve, Section~\ref{Section:3.1} provides a counterexample showing that the converse does not hold. Namely, in Example~\ref{ex:counterexample}, we construct a driving function $\lambda \in C^{0}\!\left([0,2];\mathbb{C}\right)$, smooth except at time $1$, such that the generated curve is simple for all $0 \leq t \leq 2$. However, for $0 \leq s \leq 1$, we have 
    \[
    L_{1,1+s} \cap R_1 = \ell_s = \bigg\{\frac{3\sqrt{u}}{2} \, e^{i\pi/4 u} \, \colon \,0 \leq u \leq s\bigg\}.
    \] 
    Consequently, even in the simplest case where the left hull is a two-sided curve, the condition is no longer necessary. This motivates us to introduce an alternative necessary and sufficient condition in Lemma~\ref{lem:characterization_simple}.
\end{itemize}
We now introduce the relevant definitions, building on the work of~\cite{tran2017loewner}. In the following, we always consider $a,b \geq 0$ and denote by $a \wedge b:=\max(a,b)$. We start with the following definition.
\begin{defn}\label{Def:frontier time}
    We say that a driving function $\lambda$ is \emph{top frontier} (resp. \emph{bottom frontier}) \emph{at time $t \geq 0$}, if the following limit from above (resp. below) exists:
    \[
    \lim_{\varepsilon \downarrow 0} g^{-1}_t(i\varepsilon + \lambda(t)), \qquad \text{ resp. } \qquad  \lim_{\varepsilon \uparrow 0}, g^{-1}_t(i\varepsilon+\lambda(t)).
    \]
    We say that a driving function is \emph{frontier up to times $(a,b)$}, if it is top frontier for all $0 \leq t \leq b$ and bottom frontier for all $0 \leq t \leq a$. We say that a driving function is \emph{frontier up to time $t$}, if it is frontier up to times $(t,t)$. We say that a driving function is \emph{frontier}, if it is frontier up to time $t$ for all times $t \geq 0$.
\end{defn}
We are interested in the points that are just added at a given time instant $t \geq 0$, which we are going to denote by pioneer points.
\begin{defn}[Pioneer point]\label{Def:Pioneer Point}
    We say that a point $z \in \mathbb C$ is a \emph{pioneer point at time $t \geq 0$ for the driving function $\lambda$}, if $z \in L_{t} \setminus \cup_{0 \leq s < t} L_{s}$.
\end{defn}
We are usually omitting the references to the driving function, if it is clear from the given context. Next, we want to establish the connection between the frontier limits in Definition~\ref{Def:frontier time} and two-sided curves.
\begin{defn}[Creates a two-sided curve]\label{Def:Two-sidedcurve}
    We say that $\lambda$ creates a \emph{two-sided curve up to times $(a,b)$}, if $\lambda$ is frontier up to times $(a,b)$ and there exists a curve $\gamma^{+} \colon [0,b] \to \mathbb C$ (resp. $\gamma^{-} \colon [-a,0] \to \mathbb C$) satisfying
    \[
    \gamma^{+}(t) = \lim_{\epsilon \downarrow 0} g^{-1}_t(i\epsilon+\lambda(t)), \qquad \text{resp.} \qquad \gamma^{-}(-t)=\lim_{\epsilon \uparrow 0} g^{-1}_t(i\epsilon+\lambda(t)),
    \]
    for each $0\leq t \leq b$ (resp. $0 \leq t \leq a$). We say that $\lambda$ \emph{creates a two-sided curve up to time $t \geq 0$}, if $\lambda$ creates a two-sided curve up to times $(t,t)$. We say that $\lambda$ creates a two-sided curve, if it does for all $t \geq 0$. Lastly, we say that \emph{$\lambda$ creates a two-sided curve up to times $(a,b)$ that is pioneer} if for all $-a\leq t \leq b$, the point $\gamma(t)$ is pioneer at time $|t|$.
\end{defn}
We are now prepared to relate the left hulls $(L_t)_{t \geq 0}$ directly to a generated two-sided curve.

\begin{defn}[Left hull is given by a curve]\label{Def:Two-sided-pioneercurve}
    We say that a left hull is \emph{given by a two-sided curve up to times $(a,b)$}, if $\lambda$ creates a two-sides curve up to times $(a,b)$ such that $L_{a \wedge b}=\gamma[-a,b]$. We say that a left hull is \emph{given by a two-sided curve up to time $t$}, if $t=a=b$. We say that a left hull is \emph{given by a two-sided} curve if it is given by a two-sided curve for all $t \geq 0$.
\end{defn}

In the chordal real case, if $\lambda \colon [0,\infty) \to \mathbb R$ satisfies $\|\lambda\|_{1/2} < 4$, where $\|\cdot\|_{1/2}$ denotes the Hölder-1/2 semi-norm, then there exists a quasiarc $\gamma \colon [0,\infty) \to \overline{\mathbb H}$ such that for each $t \geq 0$, one has $K_{t} = \gamma[0,t]$, where we refer to~\cite{rohde2018loewner} for an elementary proof. The following proposition is the one of the main statements of~\cite{tran2017loewner}, and can be seen as a direct generalization of the above fact to complex-driven Loewner evolutions, relating the Hölder-1/2 class to two-sided quasiarcs.
\begin{prop}\label{Thm:Tran's Theorem}
Let $\lambda \in C^{0}\left([0,\infty);\mathbb C\right)$ be a continuous driving function. There exists some $\sigma > 0$, such that if for some $T \geq 0$ one has $\|\lambda \big \vert_{[0,T]}\|_{1/2} \leq \sigma$, then the left hull is given by a two-sided quasiarc $\gamma \colon [-T,T] \to \C$, so that for each $t \in [0,T]$ one has $L_t=\gamma[-t,t]$ together with
\[
\gamma(s) = \lim_{\epsilon \downarrow 0} g^{-1}_s(i\epsilon+\lambda(s)) \qquad \text{ and } \qquad \gamma(-s) = \lim_{\epsilon \uparrow 0} g^{-1}_s(i\epsilon+\lambda(s)),
\]
where $0 \leq s \leq t$.
\end{prop}
\begin{proof}
    See~\cite[Thm~1.2]{tran2017loewner}.
\end{proof}
We give some quick remarks on this result.
\begin{itemize}
    \item Using the above definitions, one can reformulate Proposition~\ref{Thm:Tran's Theorem} as follows. If there exists some $T>0$ such that $\|\lambda \big \vert_{[0,T]}\|_{1/2} < \sigma$, then up to time $T$ the left hull is given by a two-sided curve in the sense of Definition~\ref{Def:Two-sided-pioneercurve}, which is further a quasiarc.
    \item The constant $\sigma$ is not explicitly known. Tran's proof provides a lower bound given by $\sigma > 1/3$. Using the reflection property of left hulls~(\ref{reflection:Real}) together with the theory of real-driven chordal driving functions~\cite[Theorem~2]{lind2005sharp} gives the range
    \[
    1/3 < \sigma < 4.
    \] 
    \item By considering the family of driving functions $\lambda_c(t)=c\sqrt{1-t}$ where $c \in \C$, it was shown in~\cite[Question~2]{lind2022phase} that $\sigma < 3.723$. In Lemma~\ref{lem:optimal_sigma} we are going to derive a sharper upper bound, which is given by $\sigma < 3.544$. We believe this to be non-optimal.
\end{itemize}

\subsection{Hulls driven by \texorpdfstring{$C^0$}{C^0}-drivers}\label{Section:3.1}

In this section, we work within the class of $C^0$-drivers under the additional assumption that the driver creates a left hull given by a curve. We are interested in possible generalizations of Lemma~\ref{lem:generated_by_curve}. In Example~\ref{ex:counterexample}, we prove that the condition $L_{t,t+s} \cap R_t = \{\lambda(t)\}$ for $0 \leq t \leq t+s < \infty$ is not necessarily fulfilled, even if the driving function creates a simple curve. Motivated by this example, we continue to discuss an alternative necessary and sufficient condition in Lemma~\ref{lem:characterization_simple}.

\textbf{Real driver and straight line segments.} To construct Example~\ref{ex:counterexample}, we first recall some advanced examples of real-driven chordal drivers. Namely, we review driving functions that create straight line segments emanating from the origin as chordal hulls. Fix $a \geq 0$ and define the driving function $\xi_{a}(t):=a\sqrt{t}$. In~\cite[Section~4.1]{kager2004exact}, it was shown that for each $t \geq 0$, the driver $\xi_a$ creates the following chordal hull
\begin{equation}\label{eq:angle_hulls}
    K_{t}(\xi_a(t)) = \ell_t(\Theta(a)), \qquad \text{ where } \Theta(a)= \frac{\pi}{2}\left(1-\frac{a}{\sqrt{16+a^2}}\right),
\end{equation}
where $\ell_t(\Theta(a)) \subseteq \{r e^{i \Theta(a)} \, \colon \, r \geq 0 \}$ is a line segment emanating from $\ell_0(\Theta_a) = \{0\}$. In other words, the chordal hull is given by a straight-line segment emanating from the origin with initial angle $\Theta(a)$. If we now switch to the complex-driven case, since $\xi_a$ is real valued, by the reflection property~(\ref{reflection:Real}), we have
\[
L_{t}(\xi_a) = \ell_t(\Theta(a)) \cup \ell^{\ast}_t(\Theta(a)).
\]
Thus, in the complex-driven case, the real driver $\xi_a$ grows two line segments emanating from the origin. The two lines are reflections of each other across the real axis. We are now ready to construct our example and advise the reader to consult Figure~\ref{fig:counterexample} for an illustration.
\begin{figure}[h!]\label{figure:counterexample}
    \centering
    \includegraphics[width=0.9\textwidth]{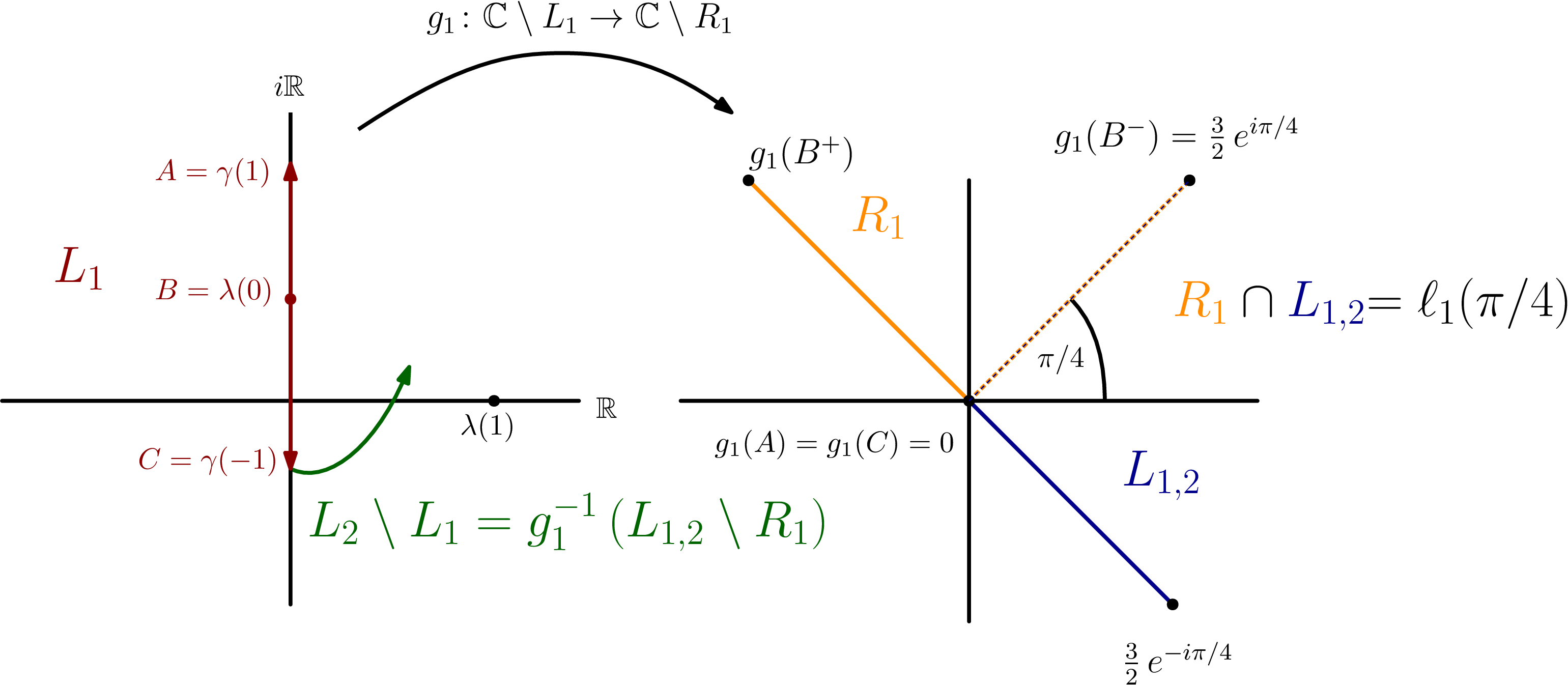}
    \caption{The total left hull given $L_2$ is given by a simple curve. The hull can be written as a disjoint union of a two-sided curve $L_1$ (\textcolor{red}{darkred}) and a negative time curve $L_2 \setminus L_1$ (\textcolor{Green}{green}). However, the intersection of $R_1$ (\textcolor{orange}{orange}) and $L_{1,2}$ (\textcolor{Blue}{blue}) is given by the line segment $\ell_1(\pi/4)$ (dashed \textcolor{orange}{orange} and \textcolor{Blue}{blue}). The prime ends of $B$ under $g_1$ from the right (resp. left) are denoted by $B^{+}$ (resp. $B^{-}$).}
    \label{fig:counterexample}
\end{figure}
\begin{example}\label{ex:counterexample} Consider the driving function $\lambda \in C^{0}\left([0,2];\mathbb C\right)$ given by 
\[
\lambda(t) = 
\begin{cases}
    \frac{4i}{\sqrt{3}}\sqrt{1-t} \quad& \text{ if } 0 \leq t \leq 1, \\
    \frac{4}{\sqrt{3}}\sqrt{t-1} \quad& \text{ if } 1 \leq t \leq 2.
\end{cases}
\]
For $t=1$ we first show that for each $0 \leq s \leq 1$, one has $L_{1,1+s} \cap R_1 = \ell_s(\pi/4)$, where $\ell_s(\pi/4) \subseteq \{r e^{i\pi/4} \, \colon \, r \geq 0\}$ is some finite length line segment that emanates from the origin with initial angle $\pi/4$.

For each $0 \leq s \leq 1$ by a simple algebraic manipulation we have
\[
L_{1,1+s} = L_s(\lambda(1+\cdot)) = L_s\left(\frac{4}{\sqrt{3}} \sqrt{\cdot}\right).
\]
By the duality principle of the right hulls~(\ref{duality:2}), the right hull is given by
\begin{equation}\label{eq:right_hull_example}
    R_{1} = i L_{1}(-i\lambda(1-\cdot)) = iL_{1}\left(\frac{4}{\sqrt{3}}\sqrt{\cdot}\right).
\end{equation}
Since we are interested in the intersection $L_{1,1+s} \cap R_1$, it now suffices to determine $L_s\left(\frac{4}{\sqrt{3}} \sqrt{\cdot}\right)$ for each $0 \leq s \leq 1$. Using~(\ref{eq:angle_hulls}) and its remarks, for each $0 \leq s \leq 1$ we have
\[
L_s\left(\frac{4}{\sqrt{3}} \sqrt{\cdot}\right) = \ell_s(\pi/4) \cup \ell^{\ast}_s(\pi/4),
\]
as a result $L_s$ is given by a finite-length line segment emanating from the origin with initial angle given by $\pi/4$ and its reflection along the real axis. Since the right hull $R_1$ is given by the same hull rotated by the imaginary unit $i$, we conclude that
\[
L_{1,s} \cap R_{1} = \ell_s(\pi/4),
\]
for each $0 \leq s \leq 1$. It is actually possible to give the explicit form of $\ell_s(\pi/4)$. Performing some algebraic manipulation of equation~(35) given in \cite[Section~(4.1)]{kager2004exact}, the line segment can be explicitly calculated as follows
\[
\ell_s(\pi/4) = \bigg\{\frac{3}{2} \sqrt{u} \, e^{i \pi/4} \, \colon \, 0 \leq u \leq s \bigg\}.
\]
So that at the final time $s=1$ the top point of this line segment is given by $3/2 \,e^{i\pi/4}$. It is left to argue that $\lambda$ creates a left hull given by a two-sided pioneer curve up to times $(2,1)$ in the sense of Definition~\ref{Def:Two-sided-pioneercurve}.

 We first prove that the left hull is given by a two-sided pioneer curve up to time $1$. 
 \begin{itemize}
     \item By similar means as above, for $0 \leq t \leq t+s < 1$ it can be proven that $L_{t,t+s}=L_s(4i/\sqrt{3}\,\sqrt{1-t+\cdot})$ such that by the remarks of Definition~\ref{Def:Left/Right_hulls} the left hull is entirely contained on the imaginary axis.
     \item Similarly it can be proven that for each $0 \leq t <1$, the right hull $R_t$ is entirely contained on the real-line, such that $L_{t,t+s} \cap R_t = \{\lambda(t)\}$ for each $0 \leq t \leq t+s <1$. As a result, using Corollary~\ref{cor:1} (note that we did not use the theory of $C^0$-drivers to prove this corollary or alternatively use~\cite[Chapter~4.1]{kager2004exact} and the reflection property~(\ref{reflection:Real})), the driver $\lambda$ creates a left hull that is given by a two-sided pioneer curve for all times $0 \leq t \leq t+s < 1$.
     \item We can strengthen this result to $t+s=1$, by using the fact that $g_1$ extends continuously to the tips of the two-sided curve and applying the concatenation principle~(\ref{concat:1}). As a result $L_1 = \gamma[-1,1]$ is a two-sided pioneer curve with starting point $B:=\gamma(0)= \{4i/\sqrt{3}\}$ and tips $A:=\gamma(1)$ and $C:=\gamma(-1)$ as presented in Figure~\ref{fig:counterexample}.
 \end{itemize}
 Next, to complete the proof that the left hull is given by a pioneer two-sided curve up to times $(2,1)$. We need to show that the increment of the left hull between time one and two is purely given by its negative times curve consisting of the bottom limits of the given driving function.
 \begin{itemize}
     \item By the concatenation principle of the left hulls~(\ref{concat:1}), we know that $L_{2} \setminus L_1 = g^{-1}_1(L_{1,2} \setminus R_1)$, where we already established that $L_{1,2}=L_1\left(4/\sqrt{3} \sqrt{\cdot}\right)$ and $L_{1,2} \setminus R_1 = \ell_1(\pi/4)$. Using the real-driven chordal case in combination with the reflection property~(\ref{reflection:Real}), we know that $L_{1,2}$ is given by a two-sided curve in the sense of Definition~\ref{Def:Two-sided-pioneercurve}, where only the negative time part is considered in $L_{1,2} \setminus R_1$. 
     \item Since $g^{-1}_1 \, \colon \mathbb C \setminus R_1 \to \mathbb C \setminus L_1$ is conformal, the preimage of the punctured line segment must be connected and contained in the interior of $\mathbb C \setminus L_1$. As a result, since $g_1 \, \colon \, \mathbb C \setminus L_1 \to \mathbb C \setminus R_1$ also extends continuous to the tips of $L_1$, we either glue the preimage of the line segment $\ell^{\ast}_1(\pi/4)$ on the tip $A$ or $C$. A more elaborate argument involving an argument presented in~\cite[Equation~(12)]{kennedy2007fast} shows that it is mapped to $C$ and it is left to prove that each point $\gamma(-(1+t))$, where $0 \leq t \leq 1$, is given by a bottom frontier limit of the driver.
     \item Using the concatenation principle of mapping out functions~(\ref{concat:0}) and the fact that $L_{1,2}$ is generated by a curve, for each $0 \leq t \leq 1$ we have
     \[
     \gamma(-(1+t))=g^{-1}_{1}(\gamma^{\lambda(1+\cdot)}(-t)) = g^{-1}_1 (\lim_{\epsilon \uparrow 0} g^{-1}_{1,1+t}(i \epsilon+\lambda(1+t)) = \lim_{\epsilon \uparrow 0} g^{-1}_{1+t}(i\epsilon+\lambda(1+t)),
     \]
     where $\gamma^{\lambda(1+\cdot)}$ is the negative time curve associated to the line segment $\ell^{\ast}_{1}(\pi/4)=L_{1,2} \setminus R_1$. This proves the example.
\end{itemize}
\end{example}
To summarize: By the previous example, if a continuous driver $\lambda$ creates a pioneer two-sided curve, then the condition $L_{t,t+s} \cap R_t = \{\lambda(t)\}$ for $0 \leq t \leq t+s \leq T$ cannot be a necessary condition for the curve to be simple. 
The main issue is that the sets $L_{t,t+s}$ and $R_t$ overlap, so that the open set $\mathbb C \setminus L_{t,t+s} \cup R_t$ stays connected. The next lemma takes this into account and states the correct alternative necessary condition.
\begin{lem}\label{lem:characterization_simple}
    Let $\lambda \in C^{0}\left([0,\infty),\mathbb C\right)$ be continuous. Let $T>0$ such that $\lambda \big \vert_{[0,T]}$ creates a left hull that is given by a two-sided curve in the sense of Definition~\ref{Def:Two-sided-pioneercurve}. Consider the following statements:
    \begin{enumerate}
        \item\label{lem:characterization_simple_1} The two-sided curve $\gamma \colon [-T,T] \to \C$ is simple.
        \item\label{lem:characterization_simple_2} For each $0 \leq t \leq T$, both $\gamma(t)$ and $\gamma(-t)$ are pioneer points.
        \item\label{lem:characterization_simple_3} For each $0 \leq t \leq t+s \leq T$, the open set $\C \setminus \left(L_{t,t+s} \cup R_{t}\right)$ is connected.
    \end{enumerate}
    Then we have $(\ref{lem:characterization_simple_1}) \iff (\ref{lem:characterization_simple_2}) \implies (\ref{lem:characterization_simple_3})$.
\end{lem}
\begin{proof}
     We first show that~(\ref{lem:characterization_simple_1}) and~(\ref{lem:characterization_simple_2}) are equivalent. Recall that for any $- T \leq t \leq T$, the point $\gamma(t)$ is a pioneer if and only if $\gamma(t) \in L_{t} \setminus \cup_{0 \leq s < t} L_s$. Since by assumption the left hull is given by a two-sided curve, this is exactly the case if and only if $\gamma(t) \neq \gamma(\pm s)$ for each $0 \leq s<t$. As a result, $\gamma$ is simple if and only if for each $0 \leq t \leq T$ the points $\gamma(\pm t)$ are pioneer, which shows the first equivalence.

     Next we prove that~(\ref{lem:characterization_simple_1}) implies~(\ref{lem:characterization_simple_3}), let $\gamma \, \colon \, [-T,T] \to \mathbb C$ be simple and fix $0 \leq t \leq t+s \leq T$. Using the fact that the complement of a simple curve is connected, by applying the concatenation principle of the left hull~(\ref{concat:2}) we use the conformal map
     \[
     g_t \, \colon \, \mathbb C \setminus L_{t+s} \to \mathbb C \setminus \left(L_{t,t+s} \cup R_t\right),
     \]
     to conclude that $\mathbb C \setminus \left(L_{t,t+s} \cup R_t\right)$ is connected.
\end{proof}

In the following remark, we discuss why~(\ref{lem:characterization_simple_3}) cannot be taken as a sufficient condition, even if we add a strict growth property to both the positive (resp. negative) time parts.
\begin{itemize}
    \item\textbf{Why is condition~(\ref{lem:characterization_simple_3}) not sufficient?} Without a strict growth property, a two-sided curve could spend a strictly positive amount of time in a past segment, satisfying~(\ref{lem:characterization_simple_3}) while being non-simple. A simple example of such a curve would be a two-sided line segment that halts for a strictly positive amount of time in at least its positive or negative time part. Currently, it is not clear if such a curve can be created using complex-driven Loewner chains and we refer to Question~1 in Section~\ref{Sec:Section_6} for a discussion.

    Notice that if we only assume that the left hull equals a two-sided curve in terms of sets, i.e. $L_t=\gamma[-t,t]$ for each $0 \leq t \leq T$, then we simply can take the curve given in Example~\ref{ex:counterexample} and extend it on its positive time part by setting $\gamma(t):=\gamma(1)$ for each $1 < t \leq 2$. The driving function remains the same as given in Example~\ref{ex:counterexample}. As a result, the left hull is given by a non-simple curve while condition~(\ref{lem:characterization_simple_3}) is still satisfied, indicating that the assumption referring to Definition~(\ref{Def:Two-sided-pioneercurve}) should not be dropped.
    \item\textbf{What if we add a strict growth condition?} Suppose that condition~(\ref{lem:characterization_simple_3}) additionally requires that the positive (resp.\ negative) time part satisfies a strict growth property, i.e., for the positive time part we have 
    $\gamma^{+}_s \subsetneq \gamma^{+}_t$ for each $0 \leq s < t \leq T$, and denote this condition by $(3)^{\prime}$. Under the assumptions stated in Lemma~\ref{lem:characterization_simple}, we provide a counterexample to the implication $(3)^{\prime} \implies (1)$.

    Let us restrict to the chordal, real-driven case for a moment and recall that $\SLE(8)$ is a space-filling curve driven by $\sqrt{8}\,W$, where $W$ is Brownian motion. Denote by $T_F$ the stopping time at which the $\SLE(8)$ first fills a closed set of positive area. By basic $\SLE$-theory or the invariance of domain, the curve up to the strictly positive time $T_F$ is almost surely non-simple, strictly growing, and the open set $\mathbb{H} \setminus \gamma[0,T_F]$ is connected. Using the reflection principle~(\ref{reflection:Real}), the driving function $\sqrt{8}\,W \big\vert_{T_F}$ generates the left hull $L_{T_F} = K_{T_F} \cup K^{\ast}_{T_F}$. Consequently, the open set $\mathbb{C} \setminus \left(L_{T_F} \cup R_0\right)$ is connected; however, the curve given by the $\SLE(8)$-trace together with its reflection is non-simple.
\end{itemize}

In the next section, we relate the class of $C^1\left([0,\infty);\mathbb C\right)$ drivers to two-sided pioneer curves.

\subsection{Two-sided pioneer curves and \texorpdfstring{$C^1$}{C^1}-drivers}\label{Section:3.2}

Recall that in the real-driven chordal setup, each driver $\lambda \in C^1\left([0,T];\mathbb R\right)$ generates a chordal hull given by a simple curve. As observed in~\cite[Theorem~3]{lind2022phase}, the same is not true in the complex-driven case. Motivated by this observation, in Proposition~\ref{lem:simple_sufficient} we prove a sufficient condition that $\lambda \in C^1([0,\infty);\mathbb C)$ creates a two-sided pioneer curve up to finite time in the sense of Definition~\ref{Def:Two-sided-pioneercurve}. The idea can be summarized heuristically as follows.

Fix $0 \leq t \leq t+s\leq T$ and decompose the left hull $L_{t+s}$ via the concatenation principle~(\ref{concat:2}):
\[
L_{t+s} = L_t \cup  g^{-1}_t(L_{t,t+s} \setminus R_t).
\]
Since $\lambda \in C^{1}$, we can partition $[0,T]$ into sufficiently small increments so that both $L_t$ and $L_{t,t+s}$ are given by two-sided pioneer curves as in Definition~\ref{Def:Two-sided-pioneercurve}. If we further assume that $L_{t,t+s} \cap R_t = \{\lambda(t)\}$, then the above decomposition simplifies to
\[
L_{t+s} = L_t \cup g^{-1}_t(L_{t,t+s} \setminus \{\lambda(t)\}).
\]
Noting that $L_{t+s} \setminus L_t = g^{-1}_t(L_{t,t+s} \setminus R_t) = g^{-1}_t(L_{t,t+s} \setminus \{\lambda(t)\})$ by the concatenation principle~(\ref{concat:2}), and using Lemma~\ref{lem:regularity_mappingout}, which shows that the tips of $L_t$ are mapped to $\lambda(t)$, we conclude that applying $g^{-1}_t \, \colon \, \mathbb C \setminus R_t \to \mathbb C \setminus L_t$ glues together a two-sided pioneer curve.

First, we prove that any finite time interval can be subdivided into sufficiently small sub-intervals on which Proposition~\ref{Thm:Tran's Theorem} can be applied.
\begin{lem}\label{lem:delta_subdivision}
    Let $\lambda \in C^1\left([0,\infty);\C\right)$. Then for each $T \geq 0$, there exists a $\delta > 0$, such that for any subdivision of $[0,T]$ into 
    \begin{equation}\label{eq:subdivision}
        0 \leq t_0 \leq t_1 \leq \dots \leq t_K = T \qquad \text{ with } \qquad 0<|t_k - t_{k-1}| \leq \delta,
    \end{equation}
    one has $\big\|\lambda(t_{k-1}+\cdot)\big\vert_{[0,t_k-t_{k-1}]} \big\|_{1/2} < \sigma$ for each $k \in \{1,\dots,K\}$, where $\sigma$ is the constant of Proposition~\ref{Thm:Tran's Theorem}.
\end{lem}
\begin{proof}
   Denoting by $\Delta_k = t_{k}-t_{k-1}$, we need to find $\delta > 0$ such that for all $k \in \{1,\dots,K\}$ we have
    \[
    \sup_{t,s \in [0,\Delta_k]} \frac{|\lambda(t_{k-1}+t)-\lambda(t_{k-1}+s)|}{|t-s|^{1/2}} \leq \sigma,
    \]
    where $\Delta_k < \delta$. Since $\lambda \in C^1([0,T];\mathbb C)$ by the intermediate value theorem for each $k$ we have 
    \[
    |\lambda(t_k+t)-\lambda(t_k+s)| \leq \sup_{\xi \in [0,\Delta_k]} |\lambda^{\prime}(t_k+\xi)| \, |t-s| \leq \sup_{\xi \in [0,T] } |\lambda^{\prime}(\xi)| \, |t-s|.
    \]
    Since $|t-s| \leq \Delta_k \leq \delta$, a small algebraic manipulation yields
    \[
    \frac{|\lambda(t_{k}+t)-\lambda(t_k+s)|}{|t-s|^{1/2}} \leq \sup_{\xi \in [0,T]} |\lambda^{\prime}(\xi)| \, \delta^{1/2}.
    \]
    Now, we can choose $0<\delta < \left(\sigma \, \sup_{\xi \in [0,T]} |\lambda^{\prime}(\xi)|\right)^{-2}$, such that for each $k$, whenever $\Delta_k <\delta$, it holds that $\|\lambda(t_{k-1}+\cdot) \big \vert_{[0,t_{k}-t_{k-1}]}\|_{1/2} \leq \sigma$. This finishes the proof.
\end{proof}
Using the previous lemma, we are now ready to prove our sufficient condition.
\begin{prop}\label{lem:simple_sufficient}
    Let $\lambda \in C^1([0,T];\C)$ and consider a subdivision of $[0,T]$ into
    \[
    0 = t_0 < t_1 < \dots < t_K = T  \qquad \text{ with } \qquad |t_{k} - t_{k-1}| \leq \delta,
    \]
    such that $\|\lambda(t_{k-1}+\cdot) \big \vert_{[0,t_{k}-t_{k-1}]}\|_{1/2} < \sigma$ for each $k \in \{1,\dots,K\}$, where $\sigma$ is the constant of Proposition~\ref{Thm:Tran's Theorem}. If the subdivision satisfies 
    \begin{equation}\label{eq:sufficient_equation}
        L_{t_k,T} \cap  R_{t_{k-1},t_k} =  \{\lambda(t_k)\},
    \end{equation}
    for each $k \in \{1,\dots,K\}$, then $\lambda$ creates a left hull that is given by a two-sided pioneer curve up to time $T$ in the sense of Definition~\ref{Def:Two-sided-pioneercurve}.
\end{prop}
\begin{proof}[Proof of Proposition~\ref{lem:simple_sufficient}]
Fix $T \geq 0$ and consider a subdivision
of $[0,T]$ given by Lemma~\ref{lem:delta_subdivision} that satisfies the additional assumption~(\ref{eq:sufficient_equation}). For each $k \in \{1,\dots,K\}$ using~(\ref{concat:2}) we have
\begin{equation*}\label{eq:glued_hulls}
    L_{t_{k-1},T} = L_{t_{k-1},t_k} \cup g^{-1}_{t_{k-1},t_k}\left(L_{t_k,T} \setminus R_{t_{k-1},t_k}\right),
\end{equation*} 
where we recall that $L_{t_{k-1},t_k}$ denotes the left hull driven by $\lambda(t_{k-1}+\cdot)$ up to time $t_{k}-t_{k-1}$ and $g_{t_{k-1},t_k} \, \colon \, \mathbb C \setminus L_{t_{k-1},t_k} \to \mathbb C \setminus R_{t_{k-1},t_k}$ is the conformal such that by~(\ref{concat:1}) we have $g_{t_{k-1},t_k}(L_{t_{k-1},T} \setminus L_{t_{k-1},t_k}) = L_{T-t_{k}}(\lambda(t_{k}+\cdot))$.

Using our additional assumption, the above reads
\begin{equation}\label{eq:gluing}
    L_{t_{k-1},T} = L_{t_{k-1},t_k} \cup g^{-1}_{t_{k-1},t_k} \left(L_{t_{k},T} \setminus \{\lambda(t_k)\}\right),
\end{equation}
Using backwards iteration on $k\in\{1,\dots,K\}$, we are going to prove that $L_T$ is given by a two-sided pioneer curve.

The base case $k=K$ follows directly since our choice of subdivision allows us to apply Proposition~\ref{Thm:Tran's Theorem} to show that $L_{t_{k-1},K}$ is given by a two-sided pioneer curve. Next, assume that $L_{t_k,T}$ is given by a two-sided pioneer curve, we prove the same is true for $L_{t_{k-1},T}$. Using the concatenation principle~(\ref{concat:2}), we have
\[
L_{t_{k-1},T} = L_{t_{k-1},t_k} \cup g^{-1}_{t_{k-1},t_k} \left(L_{t_k,T} \setminus \{\lambda(t_k)\}\right).
\]
Again using our choice of subdivision and Proposition~\ref{Thm:Tran's Theorem} the left hull $L_{t_{k-1},t_k}$ is given by a two-sided pioneer curve that starts at $\lambda(t_{k-1})$. Since by (\ref{concat:1}), the above union is disjoint and gives $L_{t_{k-1},T}$, we already know that $L_{t_{k-1},T}$ is given by a two-sided pioneer curve up to time $t_{k}-t_{k-1}$. It is left to prove that the same is true for $t_{k}-t_{k-1} < t \leq T-t_{k-1}$. For each $0<s \leq T-t_k$ denote by $\gamma_{k-1}(t_{k}-t_{k-1}+s):=g^{-1}_{t_{k-1},t_k}(\gamma^{\lambda(t_k+\cdot)}(s))$. Without loss of generality, it suffices to check that for each $0 <s \leq T-t_k$ the point $\gamma_{k-1}(t_k-t_{k-1}+s)$ is given by the top frontier limit of the driving function $\lambda(t_{k-1}+\cdot)$ at time $t_k-t_{k-1}+s$. By the concatenation principle~(\ref{concat:0}) for fixed $0<s\leq T-t_{k}$, by the above decomposition
\begin{align*}
    \gamma_{k-1}(t_{k}-t_{k-1}+s) &= g^{-1}_{t_{k-1},t_k}\left(\gamma^{\lambda(t_{k}+\cdot)}(s)\right) \\
    &=g^{-1}_{t_{k-1},t_k} \left(\lim_{\varepsilon \downarrow 0} g^{-1}_{t_k,T}(i\epsilon+\lambda(t_k+s)) \right) = \lim_{\epsilon \downarrow 0} g^{-1}_{t_{k-1},T}(i\epsilon+\lambda(t_{k-1}+(t_k-t_{k-1}+s)),
\end{align*}
which shows that $\gamma_{k-1}(t_{k}-t_{k-1}+s)$ is the top frontier limit of the driving function $\lambda(t_{k-1}+\cdot)$ at time $t_k-t_{k-1}+s$. It follows that $\Gamma_{k-1}$ is given by a two-sided pioneer curve, and this finishes the proof.
\end{proof}
The following corollary is used in Section~\ref{Section:5}. Briefly speaking, it proves that we also create a left hull given by a two-sided pioneer curve if we start with the stronger assumption that we have control over the entire intersection $L_{t,t+s} \cap R_t = \{\lambda(t)\}$, where $0 \leq t \leq t+s \leq T$.
\begin{cor}\label{cor:1}
    Let $\lambda \in  C^{1}\left([0,T];\mathbb C\right)$ and suppose that  $L_{t,t+s} \cap R_{t} = \{\lambda(t)\}$ for each $0 \leq t \leq t+s \leq T$. Then $\lambda$ creates a left hull given by a two-sided pioneer curve up to time $T$ in the sense of Definition~\ref{Def:Two-sided-pioneercurve}.
\end{cor}
\begin{proof}
    Fix $T \geq 0$, then by Lemma~\ref{lem:delta_subdivision} there exist a $\delta > 0$ and a finite subdivision of $[0,T]$, such that the claim follows if we show that the assumption~(\ref{eq:sufficient_equation}) of Proposition~\ref{lem:simple_sufficient} is satisfied. Recall that the concatenation property~(\ref{concat:3}) of right hulls states $R_{s+t}=R_{t,t+s} \cup g_{t,t+s}(R_t \setminus L_{t,t+s})$. For each $k \in \{1,\dots,K\}$, take $t:=t_{k-1}$ and $t+s:=t_{k}$, then $R_{t_k} = R_{t_{k-1},t_k} \cup g_{t_{k-1},t_k}\left(R_{t_{k-1}} \setminus L_{t_{k-1},t_{k}} \right)$ and so condition~(\ref{eq:sufficient_equation}) is automatically satisfied by our assumption. Applying Proposition~\ref{lem:simple_sufficient} finishes the proof.
\end{proof}
\section{Classification of the complex linear driver}\label{Section:4}

This section is motivated by the works of~\cite{lind2022phase} and \cite[Section~3]{kager2004exact}. Briefly, speaking in~\cite[Section~3]{kager2004exact} it was proven that the chordal hulls $(K_t)_{t \geq 0}$ driven by the driving function $\lambda(t):=\kappa t$, where $\kappa \in \mathbb R$, are generated by a curve in the sense of Lemma~\ref{lem:generated_by_curve} and its proceedings. Moreover, for sufficiently large $t \geq 0$, the growth rate of the associated curve $\gamma_{\kappa}$ was calculated. 
In~\cite[Theorem~3]{lind2022phase} it was proven that the left hulls associated with the complex-valued one-parameter family of square root drivers given by $\{c\sqrt{1-\cdot}\}_{c \in \mathbb C}$, admit a phase transition in terms of $c \in \mathbb C$. For our setup, for each $c \in \mathbb C$ define $\lambda_c(t):=ct$ as the complex linear driver. We focus on the following two key questions.
\begin{enumerate}
    \item Since there is no generated by a curve statement in the complex-driven theory, is it true that each driver $\lambda_c$ creates a left curve given by a two-sided curve up to time $t$ in the sense of Definition~\ref{Def:Two-sided-pioneercurve}?
    \item If we can answer the first question positively, does the geometry of the curve admit a phase transition in terms of $c$? If so, can we classify the different geometries in terms of $c$?
\end{enumerate}
Using our general framework developed in Section~\ref{Section:3}, in Theorem~\ref{thm:classification_linear_complete} we prove that each $\lambda_c$ creates a left hull that is given by a two-sided curve. The geometry of the underlying curve admits a phase transition in terms of $c$ and can be classified into a simple,  simple with one-end spiraling and a new exotic variant, where the negative time part stops growing after a $c$-depended finite amount of time. We are able to compute the asymptotic growth in the first two phases and determine the phase boundaries in terms of an explicit quantity involving $c$.

\subsection{The pioneer equation}
From now on, for each $c \in \mathbb C$, we denote by $\lambda_c(t):=ct$. Since we are interested in the family of hulls created by $\{\lambda_c\}_{c \in \mathbb C}$, we first need to understand how the geometry of the associated hull interplays with the complex parameter $c$. Understanding this relationship is the main task of this section.
Fix any $c \in \mathbb C \setminus \{0\}$, in Lemma~\ref{lem:calc_1} we prove that each pioneer point $z_c(t) \in L_t \setminus \cup_{0 \leq s <t}L_s$ is an implicit solution to an equation that depends on $c$. This equation governs the growth of the underlying left hull and is therefore referred to as the pioneer equation. Using the pioneer equation, we are able to guess a suitable classification result for the complex linear driver. We start with the following technical lemma that is needed for the derivation of the pioneer equation.
\begin{lem}\label{lem:2/c}
    For each $c \in \C \setminus \{0\}$, define $\lambda_c \colon [0,\infty) \to \C$ by setting $\lambda_c(t)=ct$. Then the following statements are true.
    \begin{enumerate}
        \item\label{lem:2/c_1} For any $T \geq 0$, one has $2/c \notin L_{T}(\lambda_c)$.
        \item\label{lem:2/c_2} For any $z \in \C$ and all $t < T_z$, one has $|f_t(z) - 2/c|>0$.
    \end{enumerate}
\end{lem}
Before starting the proof, we want to demonstrate one of the main techniques of working with complex linear drivers. Fix $s,t \geq 0$, applying the concatenation principle of left hulls~(\ref{concat:1}) at time $t$ for an increment of the hull up to time $t+s$, we obtain a translated copy of the same left hull up to time $s$ modulo right hull at time $t$. Precisely, by the concatenation principle of left hulls~(\ref{concat:1}) one has
    \[
    g_t(L_{t+s} \setminus L_{t}) = L_{s}(\lambda_c(t+\cdot)) \setminus R_{t},
    \]
    Using $\lambda_c(t+\cdot)=ct+\lambda_c(\cdot)$ and the translation property~(\ref{trans:0}), we rewrite as follows
    \[
    g_t(L_{t+s} \setminus L_{t}) = \left(tc+L_{s}\right) \setminus R_{t}.
    \]
The right-hand side is exactly our self-similar copy up to time $s$ that is translated by $ct$ and modulo the right hull at time $t$. We are now ready for the proof.
\begin{proof}[Proof of Lemma~\ref{lem:2/c}]
    We start by proving~(\ref{lem:2/c_1}) via a contradiction argument. Suppose $T_{2/c} \leq C < \infty$ then by Definition~\ref{Def:complex_loewner_chains} and its remarks, there exists a unique solution $(g_t(2/c))_{0 \leq t < T_z}$ to the Loewner equation
    \begin{align*}
        \partial_t g_t(2/c) &= \frac{2}{g_t(2/c)-ct}, \qquad   g_0(2/c) = 2/c,
    \end{align*}
    where $T_{2/c} >0$ is strictly positive. Solving the above gives $g_t(2/c)=2/c+ct$ and since $T_{2/c} \leq C < \infty$, applying Lemma~\ref{lem:regularity_mappingout} we have
    \[
     \lim_{t \uparrow T_{2/c}} |g_t(2/c)-ct| = 0,
    \]
    which is impossible since $c \in \C \setminus \{0\}$. For the second claim, we are going to apply a self-similarity up to translation argument. Suppose, there exist $z \in \C$ and $t \in [0,T_z)$ such that $f_t(z)=2/c$. Choosing $s > 0$ such that $0 \leq t \leq t+s<T_z$, then since $t+s<T_z$ by the concatenation principle of the left hulls~(\ref{concat:1}) one has $g_t(z) \in L_{t,t+s}\setminus R_{t}$. Applying the self-similarity of the linear driver as explained above, in terms of the centralized mapping out function we have
    \[ 
    f_t (z) \in L_s \setminus R_t.
    \]
    Consequently, $f_t(z) = 2/c \in L_{s} \setminus R_t$ and using the first part, this is impossible. This finishes the proof.
\end{proof}
We now derive the pioneer equation, which governs the growth of the underlying left hull.
\begin{lem}\label{lem:calc_1}
    For each $c \in \C \setminus \{0\}$, define $\lambda_c \colon [0,\infty) \to \C$ by setting $\lambda_c(t)=ct$. For each $z \in \mathbb C$, the centralized Loewner chains $(f_t(z))_{0 \leq t < T_z}$ satisfy
    \begin{equation}\label{eq:implicit_solution_linear}
        \frac{f_t(z)}{c}+\frac{2}{c^2} \log\left(2-cf_t(z)\right) = \frac{z}{c} + \frac{2}{c^2} \log(2-cz)-t.
    \end{equation}
    Moreover, for each $t \geq 0$, the pioneer points $z_c(t) \in L_t \setminus \cup_{0 \leq s < t} L_s$ satisfy
    \begin{equation}\label{eq:singularities_curve}
        c z_c(t)+2\log(2-cz_c(t)) = 2\log(2)+c^2 t\tag{PE},
    \end{equation}
    which we call the pioneer equation.
\end{lem}
\begin{proof}
    Fix $z \in \C$ and consider $f_t(z) = g_t(z)- \lambda_c(t)$ as in Definition~\ref{Def:complex_loewner_chains} and its remarks. For all $t < T_z$, the centralized Loewner differential equation is given by
    \[
        \frac{df_t(z)}{dt} = \frac{2}{f_t(z)} - c = \frac{2-cf_t(z)}{f_t(z)}.
    \]
    We first rewrite this in terms of a separable differential equation
    \begin{equation}\label{eq:calc1}
        f_t(z) df_t(z) = \left(2-cf_t(z)\right)dt.
    \end{equation}
    Using the previous Lemma~\ref{lem:2/c}, we are allowed to rewrite this as
    \[
     \frac{f_s(z)}{2-cf_s(z)} df_t(z) = dt.
    \]
    After applying a partial fractional decomposition this simplifies as follows
    \[
    \left(\frac{2}{c \left(2-c f_s(z)\right)} - \frac{1}{c} \right)\, df_t(z) = dt.
    \]
    Now solving this separable differential equation and applying a small algebraic manipulation yields
    \[
    -\frac{f_t(z)}{c}-\frac{2}{c^2} \log(2-cf_t(z)) = -\frac{f_0(z)}{c}-\frac{2}{c^2}\log(2-cf_0(z))+t.
    \]
    By the initial condition, we conclude that $(f_t)_{0 \leq t < T_z}$ satisfies
    \begin{equation}\label{eq:calc2}
        \frac{f_t(z)}{c} + \frac{2}{c^2} \log(2-cf_t(z)) = \frac{z}{c} + \frac{2}{c^2} \log(2-cz) - t.
    \end{equation}
    Fix $z_c(t) \in L_t \setminus \cup_{0 \leq s < t}\,L_s$ and recall that
    $\lim_{s \uparrow t} |f_s(z_c(t))|=0$ by Lemma~\ref{lem:regularity_mappingout}. By taking limits, we conclude that the pioneer point $z_c(t)$ satisfies
    \[
    cz_c(t)+2\log(2-cz_c(t)) = 2\log(2)+c^2t.
    \]
    This finishes the proof.
\end{proof}
We make the following remarks on the pioneer equation.
\begin{itemize}
    \item We denote by $\ARG \, \colon \, \mathbb C \setminus \mathbb R_{\leq 0}\to (-\pi,\pi)$ the single valued argument function, i.e. we consider the argument function under the principal branch of the logarithm, and use $\text{arg}(z)=\{\ARG(z)+2\pi k \, \colon \, k \in \mathbb Z\}$ to denote the multivalued argument function.    
    \item Throughout this work, the pioneer equation~(\ref{eq:singularities_curve}) is one of our main tools that we use to classify the complex linear driver. Considering its real and imaginary parts, we have
    \begin{align*}
    \re(c) \re(z_c(t)) - \im(z_c(t)) \im(c) + 2 \log\left(|2-cz_c(t)|\right) &= 2\log(2) + \re(c^2) t \label{eq:real}\tag{RE},\\
    \re(c) \im(z_c(t)) + \im(c) \re(z_c(t)) + 2 \text{arg}\left(2-c z_c(t)\right)  &= \im(c^2) t, \label{eq:imaginary}\tag{IM}
    \end{align*} 
    where $z_c(t)$ is some pioneer point of some fixed time instant $t \geq 0$. 
    \item Fix a pioneer point $z_c(t)$ of some time instant $t \geq 0$. By the remarks of Definition \ref{Def:complex_loewner_chains}, the centered solution $f_s(z_c(t))_{s < t}$ of the Loewner equation and its limit which solves the pioneer equation at time $t$ is unique. As a result, the pioneer point $z_c(t)$ solves \eqref{eq:imaginary} for a specific branch of the multivalued argument function. In other words, we have $\arg(2-cz_c(t))=\ARG(2-cz_c(t))+2\pi \ell(z_c(t))$, where $\ell(z_c(t))$ is some integer constant. Take for example $t=0$ and consider any pioneer point $z_c(0)$, since $L_0=\{0\}$ there is only a single pioneer point that is given by the origin such that by \eqref{eq:imaginary} we have $\ell(z_c(0))=0$. Hence, at time $t=0$ the principal branch for the argument function is considered.
    \item Fix $c \in \mathbb C$ and assume for a moment that $\lambda_c$ creates a left hull that is given by a two-sided pioneer curve in the sense of Definition~\ref{Def:Two-sidedcurve} up to time $T$. One of our goals is to trace the behavior of the created curve $\gamma_c$. By Lemma~\ref{lem:calc_1}, we know that the pioneer points are given by the positive and negative time instants of the curve, more precisely, the points $\gamma_c(\pm t) \in L_{t} \setminus \cup_{0 \leq s < t} L_{s}$ are all the pioneer points satisfying the pioneer equation at time $t$. Now for such a curve it is possible to track the crossings of the branch cut by the curve over time. As above we define $\ell \, \colon \, \gamma[-T,T] \to \mathbb Z$ pointwise by setting it to be the unique integer satisfying
    \[
    \arg(2-c\gamma_c(t)) = \ARG(2-c\gamma_c(t))+2\pi \ell(\gamma_c(t)),
    \]
    such that $\gamma_c(t)$ solves~(\ref{eq:imaginary}) for each $t \in [-T,T]$. To account for possible crossings, the value of $\ell$ changes by $1$ (resp. $-1$) if the curve $\gamma_c$ crosses the branch cut from the positive (resp. negative valued) side of the argument function $z \mapsto \ARG(2-cz)$. To get a better visualization, Figure \ref{figure:2} sketches $\ARG(2-cz)$ and the principal branch and summarizes our discussion.
    \begin{figure}[H]
    \centering
    \includegraphics[width=0.7\textwidth]{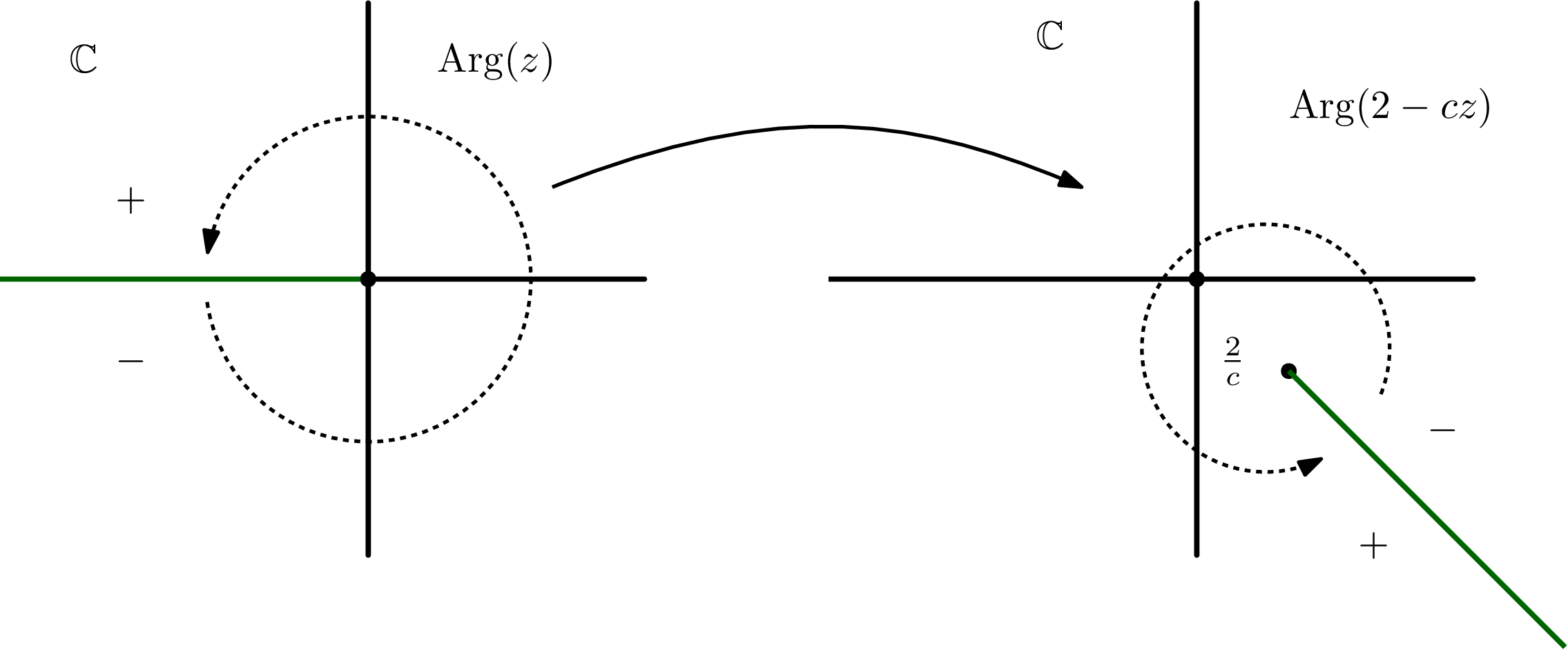}
    \caption{A sketch of $\ARG(2-cz)$ under taking the principal branch, where $c$ is an element in the first quadrant satisfying $\re(c^2)=0$. Performing a quick algebraic manipulation shows that the branch cut is the diagonal emanating from $2/c$ and crossing the forth quadrant (\textcolor{ForestGreen}{green}). Observe that points of negative (resp. positive) argument are flipped along this line rooted in $2/c$.}
    \label{figure:2}
    \end{figure}
    \item Observe that the term $\re(c^2)$ is especially interesting when it comes to the behavior of the real and imaginary part of a pioneer point for large times $t$. This suggests that any possible phase transition of the generated hulls depends on the sign of $\re(c^2)$.
\end{itemize}
We are now ready to formulate our main classification result for the complex linear driver.
\subsection{Geometric classification of the plane}
Suppose for a moment that for each $c \in \mathbb C$, the complex linear driver $\lambda_c(t)=ct$ creates a left hull that is given by a two-sided curve in the sense of Definition~\ref{Def:Two-sided-pioneercurve}. As discussed in the previous remark, if we are interested in the geometry of the created curve and its relationship with the complex parameter $c$, then it is reasonable to believe that any classification result just depends on the sign of $\re(c^2)$. We first argue that it suffices to classify the interior of the first quadrant, which we denote by $\FQ$, in terms of $\re(c^2)$.

The case $\lambda_{\kappa}(t)=\kappa t$, where $\kappa \in \mathbb R$, is already treated in~\cite[Chapter~3]{kager2004exact}. For each $\kappa$ on the real line, it was proven that $\lambda_{\kappa}$ creates a chordal hull given by a simple curve, which stays in the first quadrant. Applying the reflection principle~(\ref{reflection:Real}) shows that the same driver in the complex-driven setup creates a left hull that is given by the same curve and its reflection along the real axis. As a result, each $\lambda_{\kappa}$ creates a left hull given by a pioneer curve in the sense of Definition~\ref{Def:Two-sided-pioneercurve}. By the duality principle~(\ref{duality:2}) and similar means, the case $\lambda_{\kappa}(t)=i\kappa t$, where $\kappa \in \mathbb R$ follows analogously. So let us focus on the interior of the first quadrant, which we split into two regions and one line segment given by
\[
\Omega_+ := \{ c \in \FQ : \re(c^2) > 0 \}, \quad
\Omega_- := \{ c \in \FQ : \re(c^2) < 0 \}, \quad
\Omega_0 := \{ c \in \FQ : \re(c^2) = 0 \}.
\]
Using the reflection properties stated in Lemma~\ref{lem:reflections}, 
it suffices to understand the geometry of the two-sided curve on each set. For instance, denote by $\mathcal{C}_{+} := \Omega_{+} \cup (\Omega_{+})^{\ast} \cup (-\Omega_{+})^{\ast} \cup (-\Omega_{+})$ the reflection of $\Omega_+$ on the real (resp. imaginary) axis and the origin. Then using the reflection properties~(\ref{reflection:Real}), (\ref{Reflection:Imaginary}) and (\ref{Reflection:Origin}) in this order, we also get a classification for $\mathcal{C}_{+}$. As a result, if we prove that $\lambda_c$ creates a two-sided curve and understand its geometry in terms of $\Omega_+,\Omega_-$ and $\Omega_0$, we automatically get a classification of the plane minus the real (resp. imaginary axis) in terms of
\[
\mathbb C \setminus (\mathbb R \cup i \mathbb R)
= \mathcal{C}_{+} \cup \mathcal{C}_{-} \cup \mathcal{C}_{0},
\]
where $\mathcal{C}_{-},\mathcal{C}_0$ are given by the same applications of reflection principles. Lastly, we denote by $C_b(\mathbb C \setminus \gamma)$ the bounded connected component of the complement of a Jordan curve $\gamma$. We are now ready to present the classification result. Each complex linear driver $\lambda_c$ generates a left hull given by a two-sided curve that is either simple, simple with one end spiraling, or a previously unobserved exotic variant. We advice the reader to check Figure~\ref{figure:10001} for an illustration.
\begin{manualtheorem}{A}[Geometric Classification]\label{thm:classification_linear_complete}
For each $c \in \mathbb C$, define $\lambda_c \colon [0,\infty) \to \C$ by setting $\lambda_c(t):=ct$. For $c \in \mathbb C \setminus (\mathbb R \cup i\mathbb R)$, the geometry of the created left hulls can be classified into three different regions $\mathcal{C}_{+},\mathcal{C}_{-}$ and $\mathcal{C}_0$. It suffices to classify the interior of the first quadrant in terms of $\Omega_{+}, \Omega_{-}$ and $\Omega_{0}$, where the following statements are true.
\vspace{-\topsep}
\vspace{8pt}
\begin{enumerate}[label=\textbf{A.\arabic*}]
    \item\label{item:classification_nonzero}\textbf{Fig.(\ref{Subfigure:plus},\ref{Subfigure:minus})} For each $c \in \Omega_{+} \cup \Omega_{-}$, the driver $\lambda_c$ creates a left hull that is given by a two-sided pioneer curve in the sense of Definition~\ref{Def:Two-sided-pioneercurve}. Its asymptotic growth satisfies the following:
    \begin{itemize}
        \item\label{item:classification_nonzero_1} For $c \in \Omega_{+}$ the real and imaginary parts of the positive (resp. negative) time part of the created curve grow asymptotically linearly with rate given by
        {\setlength{\abovedisplayskip}{8pt}
        \setlength{\belowdisplayskip}{8pt}
        \begin{align*}
         \lim_{t \uparrow \infty} \frac{\re(\gamma_c(\pm t))}{t} &= \frac{\re(c^2)\,\re(c)+\im(c)\,\im(c^2)}{|c|^2},\\ 
         \lim_{t \uparrow \infty} \frac{\im(\gamma_c(\pm t))}{t} &= \frac{\re(c) \, \im(c^2)-\re(c^2) \,\im(c)}{|c|^2}. \\
        \end{align*}
        \item\label{item:classification_nonzero_2}
        For $c \in \Omega_{-}$ the negative time part of the curve spirals linearly counter-clockwise into $2/c$. Precisely speaking, we have
        {\setlength{\abovedisplayskip}{8pt}
        \setlength{\belowdisplayskip}{8pt}}
        \[
            \lim_{t \to \infty} \gamma(-t) = 2/c, \qquad \text{ and } \qquad \lim_{t \to \infty} \frac{2\ARG(2-c\gamma_c(-t))+2\ell(\gamma_c(-t))}{t} = \im(c^2),
        \]
        where $\ell$ is the function defined in the remarks following Lemma~\ref{lem:calc_1}.
        }
    \end{itemize}
    \item\label{item:classification_zero}\textbf{Fig.(\ref{Subfigure:zero})} For each $c \in \Omega_{0}$ denote by $t_{\ast}:=4\pi/\im(c^2)$. Then the following statements are true.
        \begin{enumerate}
        \item\label{subitem_zero:a} For all $0 \leq t < t_{\ast}$, the left hull is given by a two-sided pioneer curve $\gamma_c$ up to time $t$.
        \item\label{subitem_zero:b} For $t=t_{\ast}$, the left hull is given by a two-sided curve $\gamma_c$ up to time $t_{\ast}$. We have $\gamma^{+}_{t_{\ast}} \cap \gamma^{-}_{t_{\ast}} = \{0\}$, where $\gamma^{+}_{t_{\ast}}$ is simple and contained in the first quadrant for strictly positive times while $\gamma^{-}_{t_{\ast}}$ is a Jordan curve rooted at the origin, running through the forth quadrant and disconnecting $2/c$ from infinity.
        \item\label{subitem_zero:c} For each $t>t_{\ast}$ we have $L_t=\gamma_c[-t_{\ast},t]$, i.e. the left hull is given by a two-sided curve up to times $(t_{\ast},t)$ in the sense of Definition~\ref{Def:Two-sided-pioneercurve}. The positive time part is a pioneer curve contained in the first quadrant for strictly positive times, the negative time part remains the same as in the previous item. Hence for each $s \geq 0$, we have $L_{t_{\ast}+s} \cap C_b(\mathbb C \setminus \gamma_c[-t_{\ast},t_{\ast}+s]) = \emptyset$. In other words, after time $t_{\ast}$ the left hull disconnects an open set from infinity for arbitrary large times.
        \end{enumerate}
    \end{enumerate}
\end{manualtheorem}
\vspace{0.5cm}
\begin{figure}[H]
    \vspace{-0.5em}  
    \centering
    \begin{subfigure}[t]{0.30\textwidth}
        \includegraphics[width=\textwidth]{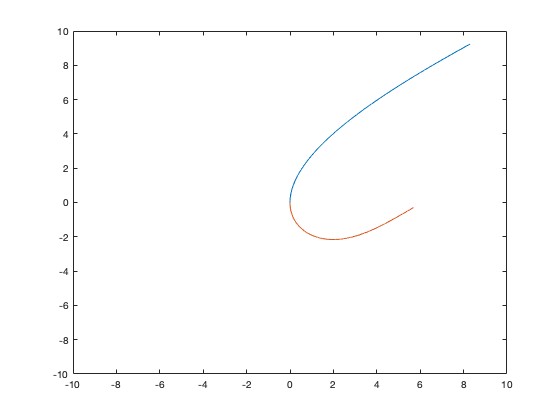}\\
        \includegraphics[width=\textwidth]{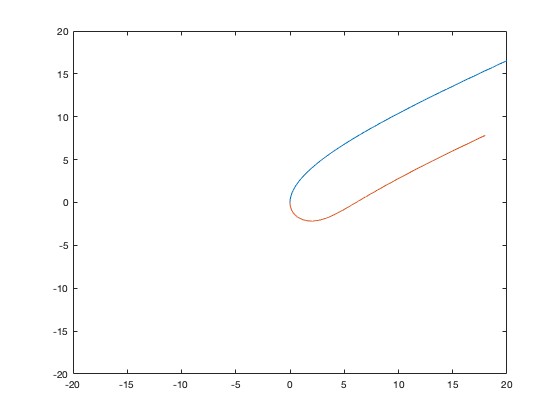}
        \caption{Simple phase for $c \in \Omega_{+}$: The left hull is given by a two-sided pioneer curve. Both the positive time part (\textcolor{blue}{blue}) and the negative time part (\textcolor{orange}{orange}) are simple.}
        \label{Subfigure:plus}
    \end{subfigure}
    \hspace{1em}
    \begin{subfigure}[t]{0.30\textwidth}
        \includegraphics[width=\textwidth]{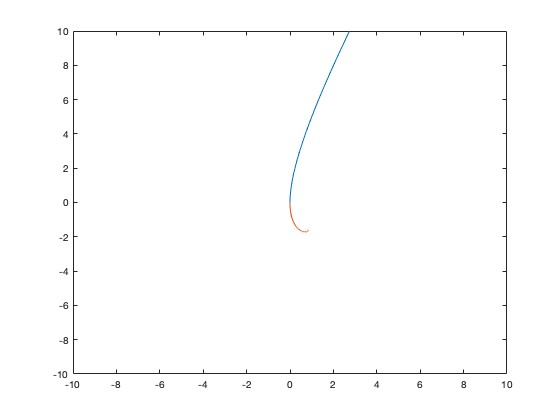}\\
        \includegraphics[width=\textwidth]{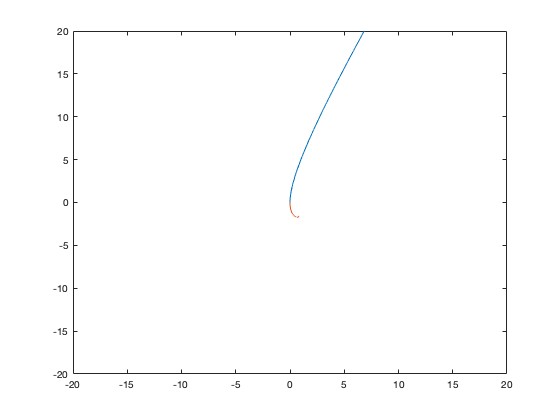}
        \caption{Simple, one arm spiraling phase for $c \in \Omega_{-}$: The left hull is given by two-sided pioneer curve. The positive time part (\textcolor{blue}{blue}) is simple, the negative time part (\textcolor{orange}{orange}) is simple but spirals counterclockwise into $2/c$ as $t \to \infty$.}
        \label{Subfigure:minus}
    \end{subfigure}
    \hspace{1em}
    \begin{subfigure}[t]{0.30\textwidth}
        \includegraphics[width=\textwidth]{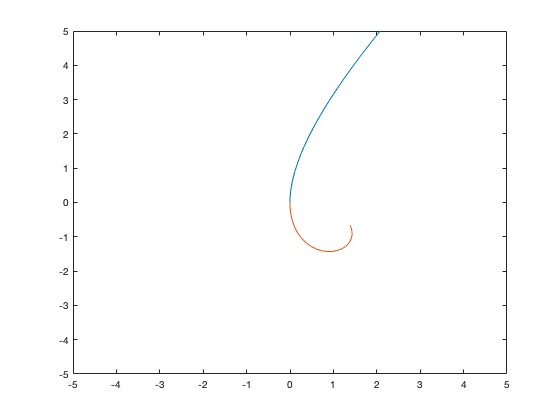}\\
        \includegraphics[width=\textwidth]{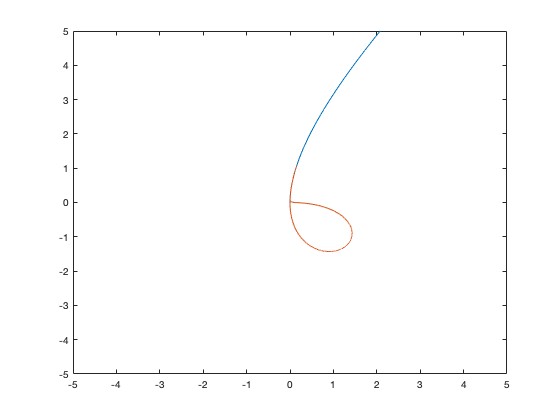}
        \caption{Exotic phase for $c \in \Omega_{0}$: For $t < t_{\ast} = 4\pi / \im(c^2)$, the left hull is given by a two-sided pioneer curve. For $t > t_{\ast}$, the positive time part (\textcolor{blue}{blue}) remains simple. The negative time part up to $t_{\ast}$ (\textcolor{orange}{orange}) forms a Jordan curve rooted at the origin. The region enclosed by the negative time curve, which is disconnected from infinity, is never added to the left hull. Consequently, the left hull continues as a two-sided curve up to times $(t_{\ast}, t)$.}
        \label{Subfigure:zero}
    \end{subfigure}
    \caption{Left hulls created by the complex linear driver $\lambda_c(t)=ct$ for $c \in \Omega_{+}, \Omega_{-}$, and $\Omega_{0}$. Each left hull is a two-sided curve, with geometry that is simple, simple with one end spiraling, or an exotic variant.}
    \label{figure:10001}
    \vspace{-0.5em}
\end{figure}
\section{Proof of the geometric classification}\label{Section:5}
In this section, we prove our geometric classification stated in Theorem~\ref{thm:classification_linear_complete}. We deal with the simple, one-sided spiraling regions $\Omega_{+}$ and $\Omega_{-}$ in Section~\ref{Section:5.1}. The exotic phase $\Omega_0$ is dealt with in Section~\ref{Section:5.2}. Using our results, in Section~\ref{Section:5.3} we derive a new sharper upper bound for the constant appearing in Proposition~\ref{Thm:Tran's Theorem}.
\subsection{Simple and spiraling phase}\label{Section:5.1}
We explain the main strategy for proving Theorem~\ref{item:classification_nonzero}. In contrast to the theory of real drivers, although $\lambda_c \in C^\infty\left([0,\infty);\mathbb C\right)$ we cannot a priori assume that $\lambda_c$ creates a left hull that is given by a two-sided pioneer curve in the sense of Definition~\ref{Def:Two-sided-pioneercurve}. Deriving this fact is our first task. We apply the theory of $C^1$-drivers developed in Section~\ref{Section:3.2}, especially Proposition~\ref{lem:simple_sufficient} and Corollary~\ref{cor:1} to establish this fact in Lemma~\ref{lem:construction_hull_simple}. Then using the pioneer equation~(\ref{eq:singularities_curve}) we calculate the asymptotic growth of the created curve in terms of $c$ in Lemma~\ref{lem:3}.

\begin{lem}\label{lem:construction_hull_simple}
For $c \in \Omega_{+} \cup \Omega_{-}$, define $\lambda_c \colon [0,\infty) \to \C$ by setting $\lambda_c(t):=ct$ and denote by $\widetilde{\lambda}_c(t):=ict$. The driving function $\lambda_c$ creates a two-sided pioneer curve in the sense of Definition~\ref{Def:Two-sided-pioneercurve}.
\end{lem}
\begin{proof}
Fix $c \in \Omega_+ \cup \Omega_{-}$ we prove that $L_{t,t+s} \cap R_t = \{\lambda_c(t)\}$ for each $0 \leq t \leq t+s < \infty$ so that the claim follows by applying Corollary~\ref{cor:1}. Using the fact that we are using a linear driving function, we first simplify the intersection above. Since $\lambda_c(t+\cdot)=ct+\lambda_c(\cdot)$ by applying the translation property~(\ref{trans:0}) we have $L_{t,t+s} = ct + L_s$ for each $0 \leq t \leq t+s<\infty$. By using the duality principle~(\ref{duality:2}) and some algebraic manipulations, we see that the right hull satisfies $R_t=ct+iL_{t}(\widetilde{\lambda}_c)$ for each $t \geq 0$.
Combining both manipulations, it now suffices to prove that
\[
L_{s}(\lambda_c) \cap iL_t(\widetilde{\lambda}_c) = \{0\},
\]
for each $0 \leq t \leq t+s < \infty$.

The origin lies in the intersection, and the remainder of the proof shows that it is the only such point. Fix $0 \leq t \leq t+s < \infty$ and consider any non-zero $z \in L_s(\lambda_c) \cap iL_t(\widetilde{\lambda}_c)$, then $z$ (resp. $-iz$) must satisfy the pioneer equation driven by $\lambda_c$ (resp. $\widetilde{\lambda}_c$) at time $T^{\lambda_c}_z$ (resp. $T^{\widetilde{\lambda}_c}_{-iz}$). We prove that this forces $z$ to be zero, by relating $T^{\lambda_c}_z$ with $T^{\widetilde{\lambda}_c}_{-iz}$ via the pioneer equation~(\ref{eq:singularities_curve}).

As $t \uparrow T_z$  the point $z$ satisfies the pioneer equation~(\ref{eq:singularities_curve}) given by
\begin{equation}\label{eq:FE_1}
    cz + 2 \log(2-cz) = 2\log(2) + c^2 T^{\lambda_c}_{z}.
\end{equation}
We need to derive a similar expression for $-iz$ by deriving a pioneer equation for the Loewner equation driven by $\widetilde{\lambda}_c$. Take any $w \in \C$ and let $0 \leq t \leq T^{\widetilde{\lambda}_c}_w$, then the centered Loewner chains $(f^{\widetilde{\lambda}_c}_t(w))_{0 \leq t \leq T^{\widetilde{\lambda}_c}_w}$ satisfy
\[
    \partial_t f^{\widetilde{\lambda}_c}_t(w) = \frac{2}{f^{\widetilde{\lambda}_c}_t(w)}, \qquad f^{\widetilde{\lambda}_c}_0(w)=w.
\]
A similar calculation as in the proof of Lemma~\ref{lem:calc_1} yields
    \[
        \int^{t}_0 \frac{2}{ic\left(2-icf^{\widetilde{\lambda}_c}_s(w)\right)} + \frac{i}{c} df^{\widetilde{\lambda}_c}_s(w) = t, \qquad \text{ with } f^{\widetilde{\lambda}_c}_0(w) = w.
    \]
Now take $w=-iz$ and integrate the above shows that for $0<t<T^{\widetilde{\lambda}_c}_{-iz}$ one has
    \[
        2\log(2-icf^{\widetilde{\lambda}_c}_t(-iz)) + icf^{\widetilde{\lambda}_c}_t(-iz) = c^2t + 2\log(2-cz)+cz.
    \]
As $t \uparrow T^{\widetilde{\lambda}_c}_{-iz}$, by the remarks of Definition~\ref{Def:complex_loewner_chains} we have
    \begin{equation}\label{eq:FE_2}
        cz+2\log(2-cz) = 2\log(2) - c^2 T^{\widetilde{\lambda}_c}_{-iz}.
    \end{equation}
Combining (\ref{eq:FE_1}) with (\ref{eq:FE_2}) we have
\[
c^2 \left(T^{\lambda_c}_z + T^{\widetilde{\lambda}_c}_{-iz}\right) = 2\pi i \left(\ell^{\lambda_c}(T_z) - \ell^{\widetilde{\lambda}_c}(T_{-iz})\right),
\]
where both constants denote the unique integer such that the pioneer point $z$ (resp. $-iz$ solves the imaginary part of its corresponding pioneer equation at time $T^{\lambda_c}_z$ (resp. $T^{\widetilde{\lambda}_c}_{-iz}$) as explained in the remarks following Lemma~\ref{lem:calc_1}. Consequently, the difference $\ell^{\lambda_c}(T_z)-\ell^{\widetilde{\lambda}_c}(T_{-iz})$ is a fixed integer constant. Using $c \in \Omega_{+} \cup \Omega_{-}$, the left-hand side has a strict positive real part, which independently of the constant on the right-hand side forces $T^{\lambda}_{z}=T^{\widetilde{\lambda}_c}_{-iz}=0$. But $L_0(\lambda_c) = L_0(\widetilde{\lambda}_c)=\{0\}$, so $z=0$ and this finishes the proof.
\end{proof}
Now that we know that for $c \in \Omega_{+} \cup \Omega_-$, the driver $\lambda_c$ creates a a left hull given by a two-sided pioneer curve, we can use the pioneer equation to calculate its asymptotic growth.
\begin{lem}\label{lem:3}
For any $c \in \Omega_{+} \cup \Omega_{-}$, define $\lambda_c \colon [0,\infty) \to \C$ by setting $\lambda_c(t)=ct$, denote by $\gamma_c$ the two-sided pioneer curve created by $\lambda_c$. Consider $\mathcal{R}_{\pm}(t):=\re(\gamma_c(\pm t))/t$ and $\mathcal{I}_{\pm}(t):=\im(\gamma_c(\pm t))/t$, then the positive (resp. negative) time parts of $\gamma_c$ satisfies the following asymptotic behavior.
\begin{enumerate}
    \item\label{lem:3.1} For $c \in \Omega_{+}$ the real and imaginary part of the positive (resp. negative) time parts of the created curve grow asymptotically linearly with rates given by:
    \begin{align*}
         \lim_{t \uparrow \infty} \mathcal{R}_{\pm}(t) = \frac{\re(c^2)\,\re(c)+\im(c)\,\im(c^2)}{|c|^2}, \qquad \lim_{t \uparrow \infty} \mathcal{I}_{\pm}(t) = \frac{\re(c) \, \im(c^2)-\re(c^2) \,\im(c)}{|c|^2}. \\
    \end{align*}
    \item\label{lem:3.2} For $c \in \Omega_{-}$ the negative time part of the curve spirals linearly counterclockwise into $2/c$, we have
    \[
    \lim_{t \to \infty} \gamma_c(-t) = 2/c, \qquad \text{ and } \qquad \lim_{t \to \infty} \frac{2\ARG(2-\gamma_c(-t))+2\pi \ell(\gamma_c(-t))}{t} = \im(c^2).
    \]
\end{enumerate}
\end{lem}
\begin{proof}
By the previous lemma, we already know that $\lambda_c$ creates a two-sided pioneer curve in the sense of Definition~\ref{Def:Two-sided-pioneercurve} with a positive (resp. negative) time pioneer point $\gamma(\pm t)$ at each time $t \geq 0$.

Fix $c \in \Omega_{+}$ and consider the real and imaginary part of the pioneer equation~(\ref{eq:singularities_curve}) given by
    \begin{align}
    \re(c) \re(\gamma_c(\pm t)) - \im(\gamma_c(\pm t)) \im(c) + 2 \log\left(|2-c\gamma_c(\pm t)|\right) &= 2\log(2) + \re(c^2) t \tag{RE}\label{eq:RE},\\
    \re(c) \im(\gamma_c(\pm t)) + \im(c) \re(\gamma_c(\pm t)) + 2 \arg\left(2-c \gamma_c(\pm t)\right) &= \im(c^2) t, \tag{IM}\label{eq:IM}
    \end{align}  
    where $\arg(2-c\gamma_c(\pm t)) = \ARG(2-c\gamma_c(\pm t)) + 2\pi \, \ell(\gamma_c(\pm t))$ as explained in the remarks following Lemma~\ref{lem:calc_1}.
We are interested in the limits of the following expressions
\begin{equation}\label{eq:limits}
    \mathcal{L}_{\pm}(t):=\frac{\log(|2-c\gamma_c(\pm t)|}{t}, \qquad \mathcal{R}_{\pm\infty}(t):=\frac{\re(\gamma_c(\pm t))}{t},  \qquad \text{ and } \qquad  \mathcal{I}_{\pm \infty}(t):= \frac{\im(\gamma_c(\pm t))}{t}.
\end{equation}
A calculation carried out in the appendix, specifically Lemma\autoref{lem:calculation_asympt}-(\ref{item:calculation_asympt_1}) shows that if $c \in \Omega_{+}$, then $\mathcal{L}_{\pm}(t) \to 0$ as $t \uparrow \infty$, while the limit for the second and third expression exists and is nonzero. Now as $t \uparrow \infty$, using~(\ref{eq:RE}) and~(\ref{eq:IM}) we get the following system of equations
\begin{align*}
        \re(c) \mathcal{R}_{\pm \infty} - \im(c) \mathcal{I}_{\pm \infty} &= \re(c^2),\\
        \re(c) \mathcal{I}_{\pm \infty} + \im(c) \mathcal{R}_{\pm \infty} &=\im(c^2).
\end{align*}
We observe that $\mathcal{R}_{\infty}$ (resp. $\mathcal{I}_{\infty}$) agrees with $\mathcal{R}_{-\infty}$ (resp. $\mathcal{I}_{-\infty}$) and solving the above system finishes the first part. For the second claim, fix $c \in \Omega_{-}$ and observe that the right hand side of~(\ref{eq:RE}) now diverges to $-\infty$ as $t \uparrow \infty$, so that the left hand side of~(\ref{eq:RE}) needs to diverge to $-\infty$ as well. By another calculation carried out in the appendix, specifically Lemma\autoref{lem:calculation_asympt}-(\ref{item:calculation_asympt_2}), we know that $\im(\gamma_c(-t)) \leq 0$, for large enough $t \geq 0$. Note that $\im(c)>0$ by $c \in \Omega_{-}$ and using the fact that $\re(\gamma_c(-t)) \geq 0$ by the remarks of Definition~\ref{Def:Left/Right_hulls}, so that the first two terms on the left hand side of~(\ref{eq:RE}) are strictly positive. Since the right hand side diverges to $-\infty$ for $t \uparrow \infty$, the first limit now follows. Using~(\ref{eq:IM}) we get the second limit and this finishes the proof.
\end{proof}
We are now ready to prove the first part of our main theorem.
\begin{proof}[Proof of Theorem~\ref{item:classification_nonzero}]
    Fix $c \in \Omega_{+} \cup \Omega_{-}$, by Lemma~\ref{lem:construction_hull_simple} the driving function $\lambda_c$ creates a left hull that is given by a two-sided pioneer curve in the sense of Definition~\ref{Def:Two-sided-pioneercurve}. Applying Lemma~\ref{lem:3} calculates the asymptotic linear growth of the real (resp. imaginary) part of the positive and negative time part of the created curve $\gamma_c$. This finishes the proof of Theorem~\ref{item:classification_nonzero}.
\end{proof}
We just finished the geometric classification for $c \in \Omega_{+} \cup \Omega_{-}$. In the next section, we deal with the exotic region $\Omega_{0}$ and prove Theorem~\ref{item:classification_zero}.

\subsection{Exotic phase}\label{Section:5.2}

This section is dedicated to the proof of Theorem~\ref{item:classification_zero}. Analogously to the simple, one-sided spiraling case $c \in \Omega_{+} \cup \Omega_{-}$, the first step is to show that for each $c \in \Omega_0$, the driver $\lambda_c$ creates a left hull that is given by a two-sided pioneer curve up to time $t$, where $0 \leq t < 4\pi/\im(c^2)$. The more delicate part is to show that for $t \geq 4\pi / \im(c^2)$ the left hull stops growing in its negative time part and only grows in its positive time curve part. The strategy for the proof of Theorem~\ref{item:classification_zero} can be summarized as follows.
\begin{itemize}
    \item Recall Corollary~\ref{cor:1}: For each $T \geq 0$, such that $L_{t,t+s} \cap R_t = \{\lambda_c(t)\}$ for each $0 \leq t \leq t+s \leq T$, we know that $\lambda_c$ creates a left hull given by a two-sided pioneer curve in the sense of Definition~\ref{Def:Two-sided-pioneercurve}. As a first step, we prove in Lemma~\ref{lem:bound_tau} that $\lambda_c$ violates the above criterion for the first time at some time $\tau \geq 4\pi/\im(c^2)$. Consequently, for all times smaller than $4\pi/\im(c^2)$, the driver $\lambda_c$ creates a two-sided pioneer curve, establishing Theorem~\ref{subitem_zero:a}.    
    \item However, in contrast to the real-driven chordal case: Example~\ref{ex:counterexample} shows that the criterion of Corollary~\ref{cor:1} is only a sufficient condition to show that a driver creates a two-sided curve, which is simple. As a result, we cannot directly deduce that $4\pi/\im(c^2)$ is the smallest time, when the created two-sided curve is non-simple. A major part of this section is dedicated to this task.
    \item Under the assumption that $\lambda_c$ creates a positive time curve, in Lemma~\ref{lem:upper_part_simple} we prove that the positive time part is a simple curve contained in the upper half-plane, while starting at the origin.
    \item In Lemma~\ref{lem:100}, we prove that up to time $4\pi/\im(c^2)$, the negative time part satisfies a specific reflection principle. The reflection principle in combination with Lemma~\ref{lem:root_zero} is going to imply that the negative time part of $\gamma_c$ up to $4\pi/\im(c^2)$ is a Jordan curve rooted in the origin, which disconnects $2/c$ from infinity.
    \item Using the same reflection principle, by applying a self-similarity argument specific to the complex linear driver, we prove that the increment of the left hull after time $4\pi/\im(c^2)$ is only given by a positive time curve part.
\end{itemize}

\textbf{Positive time part.} First, assuming that $\lambda_c$ creates a positive-time curve in the sense of Definition~\ref{Def:Two-sidedcurve}, we show that it is a simple curve in the upper half-plane.
\begin{lem}\label{lem:upper_part_simple}
    For each $c \in \Omega_0$, define $\lambda_c$ by setting $\lambda_c(t):=ct$. Fix $c \in \Omega_0$ and suppose that $\lambda_c$ creates a left hull given by a two-sided curve up to times $(a,T)$, where $a,T$ are both positive. Denote by $\gamma^{+}_c$ the positive time part up to time $T$. Then $\gamma^{+}_c$ is a simple curve contained in the upper half-plane for strictly positive times.
\end{lem}
\begin{proof}
    To prove that $\gamma^{+}_c$ is simple, it suffices to show that for each $0 \leq s \leq T$ the positive time curve $\gamma^{+}_{T-s}$ is contained in upper half-plane for all strict positive times. Indeed, by Corollary~\ref{cor:1} the positive time curve $\gamma^{+}_{c}$ is simple if $\gamma^{+}_{s,T} \cap R_s = \{\lambda(s)\}$. Recall that by the translation property~(\ref{trans:0}) we have $L_{s,T}=cs+L_{T-s}$. Furthermore, by the remarks of Definition~\ref{Def:Left/Right_hulls}, the right hull is contained in the imaginary strip $R_s \subseteq \{z \in \mathbb C \, \colon \, 0 \leq \im(z) \leq \im(c) s\}$.  Consequently, since $\lambda_c$ creates a positive time curve, it suffices to show that for each $0 \leq s \leq T$ the curve $\gamma^{+}_{T-s}$ is contained in the upper half-plane for strict positive times.

    Fix $w \in \gamma^{+}_{T-s} \setminus \{\lambda(0)\}$, since $\lambda_c$ creates a two-sided curve in the sense of Definition~\ref{Def:Two-sided-pioneercurve}, the point $w$ is given by the top frontier limit
    \[
    w = \lim_{\epsilon \downarrow 0} g^{-1}_t\left(i\epsilon+\lambda(t)\right),
    \]
    for some $t \in (0,T-s]$. Using the reverse Loewner equation, we prove that $\im(w)>0$ by establishing $\im(w) \geq \im(c) t$.

    Fix $\epsilon > 0$, denote by $z:=i\epsilon+\lambda(t)$ and consider the reverse Loewner equation given by 
    \[
    \partial_u h_u(z) = \frac{-2}{h_u(z)-\lambda(t-u)}, \qquad h_0(z) = z,
    \]
    where $u \in [0,t]$. Observe that $\im(z)=\epsilon + \im(c)t>0$ by $\im(c)>0$ and $\epsilon>0$, since we take a top frontier limit. By the time reversal property stated in Lemma~\ref{lem:time_reversal} we know that $g^{-1}_{t}(z) = h_{t}(z)$ and so it suffices to bound $\im(h_t(z))$. Differentiating with respect to $u$, using $\im(c)>0$ and taking imaginary parts yields
    \[
     \partial_u \im(h_u(z)) = \frac{2\,\im\left(h_u(z)-\lambda(t-u)\right)}{|h_u(z)-\lambda(t-u)|^2} \geq \frac{2 \,\im\left(h_u(z)-tc\right)}{|h_u(z)-\lambda(t-u)|^2},
    \]
    for all $u \in [0,t]$. Now using $\im(z)>\im(c)t$, we see that $\partial_u \im(h_u(z))\big\vert_{u=0} > 0$ and since the above inequality holds uniformly in $u \in [0,t]$, we conclude that
    \[
    \im(h_u(z)) \geq \im(h_0(z)) > \epsilon+\im(c)t,
    \]
    for all $u \in [0,t]$. Applying $g^{-1}_t(z)=h_t(z)$ we have 
    \[
    \im\left(g^{-1}_t(z)\right) > \epsilon+\im(c)t,
    \]
    for each $\epsilon>0$ and so 
    \[
    \im(w) = \lim_{\epsilon \downarrow 0} \im\left(g^{-1}_t(z)\right) \geq \im(c)t.
    \]
    This proves $\im(w) \geq \im(c) t$ for each $t \in (0,T-s]$, and so $\gamma^{+}_{T-s}$ is contained in the upper half-plane for strictly positive times. This finishes the proof.
\end{proof}
We note that by combining this result with the remarks following Definition~\ref{Def:Left/Right_hulls}, it follows that any possible created positive-time curve is contained in the first quadrant for strict positive times.\newline

\textbf{Negative time part.} Now we turn our attention to the negative time part. For each $c \in \Omega_0$, the pioneer equation~(\ref{eq:singularities_curve}) is going to play a major role in the proof of Lemma~\ref{lem:bound_tau}. Consequently, under the condition that $\lambda_c$ creates a two-sided pioneer curve, we are interested to see whether the created curve crosses the branch cut given by the fourth diagonal emitted from $2/c$ as illustrated in Figure~\ref{figure:2}. The next lemma proves that the negative time part crosses the cut at time $2\pi/\im(c^2)$ and that up to time $4\pi/\im(c^2)$ there are no more crossings.
\begin{lem}\label{lem:branch_cut_calc}
    Fix $c \in \Omega_0$ and suppose that $\lambda_c$ creates a two-sided pioneer curve $\gamma_c$ up to some positive time $T$. Denote by $t_{\text{cut}}:=\inf\{0 \leq t <T \, \colon \, \re(\gamma_c(-t)) = -\im(\gamma_c(-t)) \, \text{ and } |\gamma_c(-t)| \geq 2/|c|\}$ the first time that the negative time part of $\gamma_c$ touches the branch cut of $\arg(2-cz)$. If $T \geq 2\pi/\im(c^2)$, then the following statements are true.
    \begin{enumerate}
        \item The first time $\gamma_c$ touches the branch cut is given by $t_{\text{cut}}=2\pi/\im(c^2)$.
        \item For all $t_{\text{cut}} \leq t < 4\pi/\im(c^2)$, one has $\ell(\gamma_c(-t)) = 1$.
    \end{enumerate}
\end{lem}
Before starting the proof, we want to remind the reader that since the positive time part is contained in the upper half-plane for strict positive time by Lemma~\ref{lem:upper_part_simple}, it is only possible for the negative time part of $\gamma_c$ to cross the branch cut.
\begin{proof}[Proof of Lemma~\ref{lem:branch_cut_calc}]
    Suppose that $\lambda_c$ creates a two-sided curve up to a positive time $T \geq 2\pi/\im(c^2)$. By using $c \in \Omega_0$ we have $\re(c)=\im(c)$, so the imaginary part of the pioneer equation~(\ref{eq:IM}) for the negative time pioneer reads
    \[
    \re(c)\left(\im(\gamma_c(-t))+\re(\gamma_c(-t))\right) + 2 \ARG(2-c\gamma_c(-t)) + 4\pi \ell(\gamma_c(-t)) = \im(c^2) t.
    \]
    Using the fact that $\gamma_c$ is pioneer, we are going to prove that $2\pi/\im(c^2)$ is the first time the negative time part of $\gamma_c$ crosses the branch cut.

    Consider any $t \geq 0$ such that $\re(\gamma_c(-t))=-\im(\gamma_c(-t))$, a quick algebraic manipulation of the above equation gives $t=2\pi(-1+2\ell(\gamma_c(-t)))/\im(c^2)$, where $\ell(\gamma_c(-t))\geq 1$. Since by assumption the curve $\gamma_c$ is pioneer, there exists a unique negative time point $\gamma_c(-2\pi/\im(c^2))$ that satisfies the imaginary part of the pioneer equation. So, if we fix $t=2\pi/\im(c^2)$, then the negative time solution of the above equation at time $2\pi/\im(c^2)$ is unique. Consequently, by the above calculation we have $\re(\gamma_c(-2\pi/\im(c^2))=-\im(\gamma_c(2\pi/\im(c^2))$ and $\ell(\gamma_c(-2\pi/\im(c^2))=1$ so that $t_{\text{cut}}=2\pi/\im(c^2)$.

    For the second claim, checking the imaginary part of the frontier equation, we see that for $t>t_{\text{cut}}$ a positive-negative crossing of the argument function $\ARG(2-c\gamma_c(-t))$ by the negative time part of the curve, resulting in $\ell(\gamma_c(-t))=0$ while $\ell(\gamma_c(-t_{\text{cut}}))=1$, is impossible. Furthermore, by the remarks of Definition~\ref{Def:Left/Right_hulls} we have $\re(\gamma_c(-t)) \geq 0$, so that $\ell(\gamma_c(-t)) \geq 2$ and continuity imply that there exists some time instant $t_{\text{cut}} < t_{\star} < t$, such that the associated negative time pioneer point satisfies $\re(\gamma_c(-t_{\star}))=-\im(\gamma_c(-t_{\star}))$. Using the calculation above, the earliest time after $2\pi/\im(c^2)$ for such a time instant is given by $t_{\star}=4\pi/\im(c^2)$. This finishes the proof.
\end{proof}
Next, we prove that the first time instant at which the negative time part is non-simple must be a revisit of the origin.
\begin{lem}\label{lem:root_zero}
    Suppose that $\lambda_c$ creates a two-sided curve $\gamma_c$ up to a positive time $T$ and denote by $t_{\ast}(\mathrm{r}):=\inf\{0 < t\leq T \, \colon \, \gamma(-t) = \gamma(-t_\mathrm{r}) \; \text{ for some } 0\leq t_{\mathrm{r}} < t\}$ the first revisiting time of the negative time part of $\gamma_c$. If there exists some $\mathrm{r} \in \mathbb C$ such that $t_{\ast}(\mathrm{r})$ is finite, then $\mathrm{r}=0$.
\end{lem}
The idea of the proof is to apply the self-similarity up translation property of the complex linear driver.
\begin{proof}[Proof of Lemma~\ref{lem:root_zero}]
By assumption, there exists some $\mathrm{r} \in \mathbb C$, such that $\gamma(-t_{\mathrm{r}})=\gamma(-t_{\ast}(\mathrm{r))}$, where $t_{\ast}(r)$ is finite. Suppose that $t_{\mathrm{r}}>0$ and choose $0 < \varepsilon < t_{\mathrm{r}}$, so that by applying~(\ref{concat:1}) and the translation property~(\ref{trans:0}) we have
\[
    g_{t_{\mathrm{r}}-\varepsilon}(L_{t_{\ast}} \setminus L_{t_{\mathrm{r}}-\varepsilon}) = L_{t_{\mathrm{r}}-\varepsilon,t_{\ast}} \setminus R_{t_{\mathrm{r}}-\varepsilon} = \big((t_{\mathrm{r}}-\varepsilon)+ L_{t_{\ast}-t_{\mathrm{r}}+\varepsilon} \big) \setminus R_{t_{\mathrm{r}}-\varepsilon}.
\]
By our choice of $\varepsilon > 0$ we have $0<t_{\ast}(\mathrm{r})-t_{\mathrm{r}}+\varepsilon<t_{\ast}(\mathrm{r})$. But then the left-hand side is the image of a non-simple curve under a conformal map while the right-hand side contains at most a simple curve by definition of $t_{\ast}(\mathrm{r})$ and our choice of $\varepsilon$. This is impossible, hence $t_{\mathrm{r}}=0$ and $\mathrm{r}=0$ and this finishes the first part.
\end{proof}
So far we have assumed that $\lambda_c$ creates a left hull given by a two-sided curve up to some time $T \geq 0$. We now show that the earliest time at which $\lambda_c$ can create a non-pioneer two-sided curve is given by $4\pi/\im(c^2)$. Before stating the lemma, recall that $t_{\text{cut}}$ denotes the smallest time at which a possibly created two-sided pioneer curve touches the branch cut shown in Figure~\ref{figure:2}.
\begin{lem}\label{lem:bound_tau}
    For each $c \in \Omega_0$, denote by $\widetilde{\lambda}_c(t):= ict$ and let $\tau$ be the smallest positive time such that $L_{t,t+s} \cap \left(R_t \setminus \{\lambda_c(t)\}\right)$ is nonempty for some $0 \leq t \leq t+s \leq \tau$. Then for each $c \in \Omega_0$, we have $\tau \geq 4\pi/\im(c^2)$.
\end{lem}
\begin{proof}
    Fix some $0 \leq t \leq t+s < \infty$, using $c \in \Omega_0$ we first simplify the expression $L_{t,t+s} \cap R_t = \{\lambda_c\}$. Recall that $\lambda_c(t+\cdot)=ct+\lambda_c$ so by the translation property~(\ref{trans:0}), we have $L_{t,t+s} = ct + L_s$. By the duality principle~(\ref{duality:2}), the right hull satisfies $R_t = ct+iL_t(\widetilde{\lambda}_c)$. Using $c \in \Omega_0$ we further have $\widetilde{\lambda}_c=-\lambda^{\ast}_c$, so by applying our reflection principles, the right hull can be rewritten as $R_t=ct-iL^{\ast}_t$. We now want to lower bound the smallest time $\tau$ such that $L_s \cap -iL^{\ast}_t$ contains a non-zero point $z$ for some $0 \leq t \leq t+s \leq \tau$. We prove that if there such a non-zero point $z$ in the above intersection, then $t+s \geq 4\pi/\im(c^2)$.

    Let $z$ be such a point, then by above $z$ (resp. $-iz^{\ast}$) is a pioneer point at some time $0<T_z \leq s$ (resp. $0<T_{-iz^{\ast}} \leq t$), so that it suffices to  prove $T_z + T_{-iz^{\ast}} \geq 4\pi/\im(c^2)$.

    At time $T_z$ the point $z$ satisfies the pioneer equation~(\ref{eq:singularities_curve}) given by
    \[
    c z + 2 \log(|2-cz|) + 2\ARG(2-cz) + 2\pi i \, \ell(z)=2\log(2) + c^2 T_{z},
    \]
    where we emphasize that the constant $2\pi$ instead of $4\pi$ in front of $\ell(z)$ is the correct one, since at this stage of the proof we don't know yet that the left hull is generated by a curve, hence the pioneer point could solve the imaginary part of the pioneer equation on any possible branch of the argument function as explained in the remarks of Lemma~\ref{lem:calc_1}.
    At time $T_{-iz^{\ast}}$ the point $-iz^{\ast}$ satisfies
    \[
    -ic z^{\ast} + 2\log(|2+icz^{\ast}|) + 2\ARG(2+icz^{\ast}) + 2\pi i \, \ell(-iz^{\ast}) = 2\log(2)+c^2 T_{-iz^{\ast}}.
    \]
    Observe that $c \in \Omega_0$ implies that $-ic = c^{\ast}$ and $-c^2=(c^2)^{\ast}$ and so taking the complex conjugate gives
    \[
    cz + 2\log(|2-cz|) +2\ARG(2-cz) -2\pi i \ell(-iz^{\ast}) = 2\log(2) - c^2  T_{-iz^{\ast}}.
    \]
    Subtracting the first equation from the second one and taking the imaginary part gives
    \[
    \im(c^2) (T_z+T_{-iz^{\ast}}) = 2\pi \left(\ell(z) + \ell(-iz^{\ast})\right),
    \]
    where $\ell(z)+\ell(-iz^{\ast})$ is a sum of two integer constants. We first show that $T_z \geq t_{\text{cut}}=2\pi/\im(c^2)$ and then prove that $T_z+T_{-iz} \geq 4\pi/\im(c^2)$.

    Since $\im(c^2)>0$ the left hand side is positive, so is the right hand side. If the right hand side is zero, then $T_{z}+T_{-iz^{\ast}}=0$ hence $z=0$, so the right hand side above must be at least $2\pi$. This gives $\tau \geq 2\pi/\im(c^2)$, so it is now legitimate to say that $\lambda_c$ creates a left hull given by a two-sided pioneer curve until some time greater than $2\pi/\im(c^2)$ by using Corollary~\ref{cor:1}. It is now legitimate to write $\ell(z)=\ell(\gamma_c(\pm T_z))$.
    If $\tau \geq 4\pi/\im(c^2)$ we are done, otherwise by the previous lemma we get the following equality
    \[
    2\pi \left(1+\ell(-iz^{\ast})\right) \geq 2\pi \left(\ell(\gamma_c(\pm T_z)) + \ell(-iz^{\ast})\right) = \im(c^2) \left(T_z + T_{iz^{\ast}}\right).
    \]
    We have $z \in L_{T_z} \cap -iL^{\ast}_{T_z} \setminus \{\lambda(T_z)\}$, where $z$ is nonzero. Recall that up to time $T_z$ the left hull is given by a two-sided pioneer curve. By continuity and $\re(\gamma_c( \pm t)) \geq 0$, there must be some time $0 <t_{\star} \leq T_z$ and a corresponding pioneer point such that $\re(\gamma_c(\pm t_{\star}))=-\im(\gamma_c(\pm t_{\star}))$. Using the fact that the positive time part is contained in the first quadrant by Lemma~\ref{lem:upper_part_simple}, this is only possible for a negative time pioneer point. We still need to rule out that $|\gamma_c(-t_{\star})|<2/c$, i.e. not crossing the branch cut, resulting in $\ell(\gamma_c(-t_{\star}))=0$ but using the imaginary part of the pioneer equation it is easy to see that this is impossible. As a result it must be that $z$ is a negative time pioneer point with
    $\ell(\gamma_c(-T_z)) = 1$ and so $T_z \geq t_{\text{cut}}=2\pi/\im(c^2)$.

    We now prove that $T_z+T_{-iz} \geq 4\pi/\im(c^2)$, using $T_z \geq t_{\text{cut}}$ and Lemma~\ref{lem:branch_cut_calc} we have
    \[
    2\pi \left(1+\ell(-iz^{\ast})\right) = 2\pi \left(\ell(\gamma_c(-T_z)) + \ell(-iz^{\ast})\right) = \im(c^2) \left(T_z + T_{iz^{\ast}}\right) \geq 2\pi + T_{-iz}.
    \]
    But now by using the above inequality, it must be that $\ell(-iz^{\ast})$ is also non-negative integer constant. If $\ell(-iz^{\ast})=0$, then this implies that $T_{-iz^{\ast}}=0$, hence $z=0$ and so $\ell(-iz^{\ast})$ must be at least $1$, so that $4\pi/\im(c^2) \leq T_z+T_{-iz^{\ast}} \leq \tau$. This finishes the proof.
\end{proof}
We now know that $\lambda_c$ creates a two-sided curve up to time $4\pi/\im(c^2)$ in the sense of Definition~\ref{Def:Two-sided-pioneercurve}. The next Lemma, establishes that the negative time part of the created curve satisfies a specific reflection property, while the positive time part does not.
\begin{lem}\label{lem:100}
    For each $c \in \Omega_0$, denote by $\gamma^{+}_T$ (resp. $\gamma^{-}_T$) the positive (resp. negative) time part of the left hull created by $\lambda_c$ up to time $T$ and let $t_{\ast}=4\pi/\im(c^2)$. Then the following statements about the positive and negative time parts of the curve are true.
    \begin{enumerate}
        \item\label{100.1} For each $0 \leq t \leq t_{\ast}$, the negative time part satisfies the reflection property $-i\gamma^{\ast}_{c}\left(-(t_{\ast}-t)\right)=\gamma_c(-t)$.
        \item\label{100.2} For any $T>0$, the positive time part satisfies $\gamma^{+}_T \cap -i(\gamma^{+}_T)^{\ast}=\{0\}$.
    \end{enumerate}
\end{lem}
Before starting the proof, we advise the reader to check Figure~\ref{fig:rotation} for an illustration of the stated reflection property.
\begin{figure}[h]
    \centering
    \includegraphics[width=0.8\linewidth]{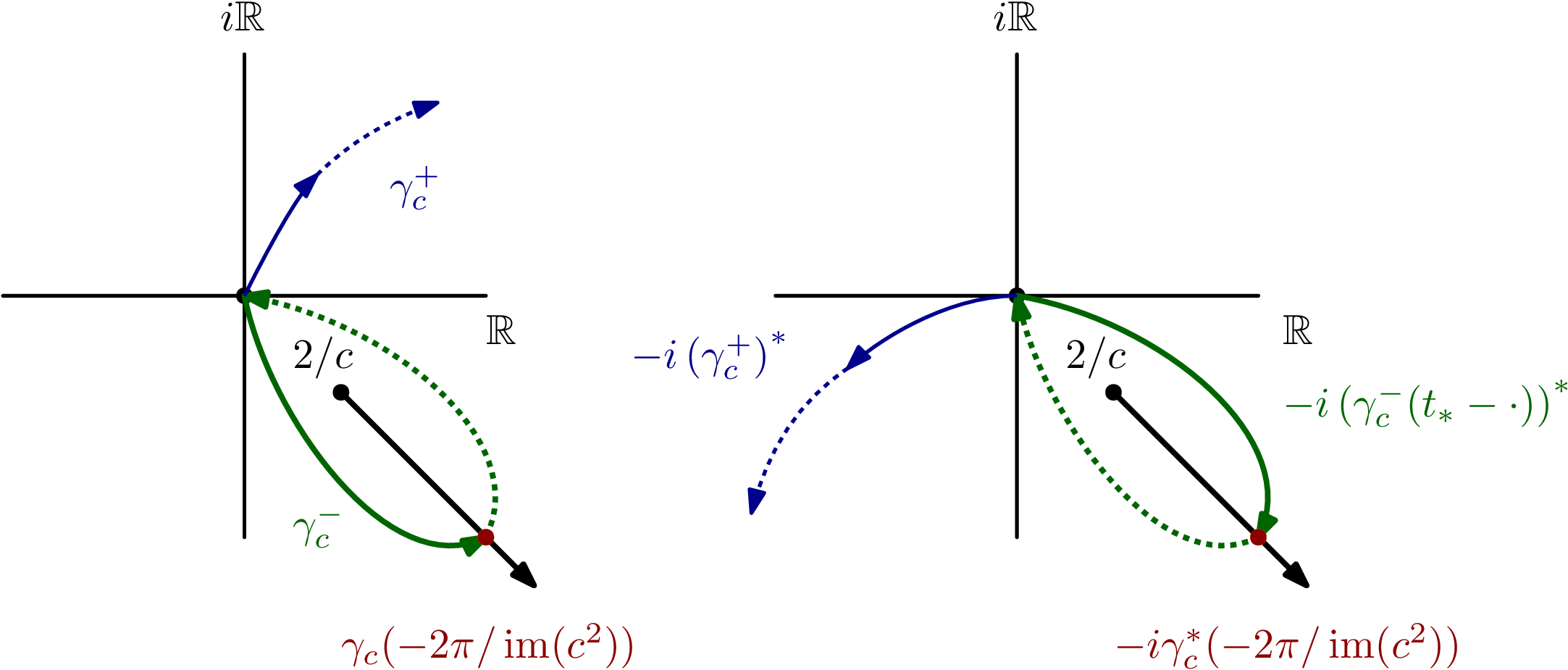}
    \caption{\textbf{Left Side:} The created two-sided curve up to time $4\pi/\im(c^2)$. The positive time part $\gamma^{+}_c$ (\textcolor{blue}{blue and dashed blue}) is contained in the first quadrant by Lemma~\ref{lem:upper_part_simple} and the remarks of Definition~\ref{Def:Left/Right_hulls}. By Lemma~\ref{lem:branch_cut_calc} the negative time part $\gamma^{-}_c$ (\textcolor{ForestGreen}{green}) crosses the branch cut at $\gamma_c(-2\pi/\im(c^2))$ (\textcolor{red}{red}) and afterwards completes a loop rooted in the origin until time $4\pi/\im(c^2)$ (\textcolor{ForestGreen}{dashed green}).\newline
    \textbf{Right side:} The positive time part under the transformation $\gamma^{+}_c \mapsto -i(\gamma^{+}_c)^{\ast}$ (\textcolor{blue}{blue and dashed blue}) and the negative time part under the transformation $\gamma_c(-t) \mapsto -i\gamma^{\ast}_c(-\left(t_{\ast}-t)\right)$ (\textcolor{ForestGreen}{green and dashed green}) up to time $4\pi/\im(c^2)$. The transformation only preserves the negative time part up to time $4\pi/\im(c^2)$ by reversing the original orientation as stated in Lemma~\ref{lem:100}.}
    \label{fig:rotation}
\end{figure}
\begin{proof}[Proof of Lemma~\ref{lem:100}]
 We start with the reflection symmetry of the negative time part given in~(\ref{100.1}). Fix $0 < t < t_{\ast}$, using the pioneer equation~(\ref{eq:singularities_curve}) we prove that $-i\gamma^{\ast}_c(-(t_{\ast}-t))=\gamma_c(-t)$. We simplify our notation by fixing $w:=\gamma_c(-(t_{\ast}-t))$ and $z:=\gamma_c(-t)$. We want to prove that $-iw^{\ast}=z$. Without loss of generality, take $0<t \leq t_{\ast}/2$ so that the point $w$ satisfies the pioneer equation at time $t_{\ast}-t$ given by
    \begin{equation}\label{eq:pioneer_w}
        cw+2\log(2-cw)+4\pi i = 2\log(2) + c^2 (t_{\ast}-t).
    \end{equation}
    The additional factor must be $4\pi i$, because $\lambda_c$ creates curve up to time $t_{\ast}$ where the condition $t_{\ast}-t \geq 2\pi/\im(c^2)=t_{\text{cut}}$ implies that the curve associated to $w$, already crossed the branch cut as stated in Lemma~\ref{lem:branch_cut_calc}. We prove that $-iw^{\ast}$ satisfies the pioneer equation of $z$ at time $t$, which is given by
    \[
    cz+ 2\log(2-cz) = 2 \log(2) + c^2 t,
    \]
    where the additional constant satisfies $\ell(z)=0$ since the curve associated to $z$ has not crossed the branch cut since $t<t_{\ast}/2=t_{\text{cut}}$. Taking the complex conjugate of~(\ref{eq:pioneer_w}) gives
    \[
    c^{\ast} w^{\ast} + 2\log(2-cw)^{\ast} - 4\pi i= 2\log(2)+(c^2)^{\ast} (t_{\ast}-t),
    \]
    By $c \in \Omega_0$, we have $c^{\ast}=-ic$ and $(c^2)^{\ast}=-c^2$ so that
    \[
    -icw^{\ast} + 2\log(2+icw^{\ast}) - 4\pi i = 2 \log(2) - c^2 (t_{\ast}-t).
    \]
    Observe that $-c^2 t_{\ast} = -4\pi i$, so that
    \[
    -icw^{\ast} + 2\log(2+icw^{\ast}) = 2 \log(2) + c^2 t.
    \]
    This is exactly the pioneer equation of $z$ at time $t$. Since we already know that we generate a two-sided pioneer curve up to time $t<t_{\ast}$, there is at most one negative time pioneer point at time $t$ and so $-iw^{\ast}=z$. For $t=t_{\ast}$, suppose that $\gamma_c(-t_{\ast})$ is a pioneer point, then it satisfies the pioneer equation and we can apply the same procedure as above, which shows $\gamma_c(-t_{\ast})=-i\gamma_c(0)^{\ast}=0$. As a result, $\gamma_c(-t_{\ast})$ cannot be a pioneer point, so that at time $t_{\ast}=4\pi$, the negative time curve is non-simple. Applying Lemma~\ref{lem:root_zero} proves that the curve revisits the origin. This finishes 
    the first part.

    For the second claim fix $T \geq 0$, since it is clear that the origin is an element of the intersection $\gamma^{+}_T \cap -i(\gamma^{+}_T)^{\ast}$, it is left to prove it is the only point. If not, there exist distinct $s,t \in (0,T]$ such that $\gamma^{+}_c(t) = -i(\gamma^{+}_c(s))^{\ast}$ and a quick algebraic manipulation shows that $\re(\gamma^{+}(t))=-\im\left(\gamma^{+}_c(s)\right)$. By Lemma~\ref{lem:upper_part_simple}, we have $\im(\gamma_c(s))>0$ for $s>0$, so using the fact that $\re(\gamma_c(t)) \geq 0$ for $t \geq 0$ by the remarks of Definition~\ref{Def:Left/Right_hulls}, this is impossible. This finishes the proof.
\end{proof}
So far we have shown that, up to time $4\pi/\im(c^2)$ the driver $\lambda_c$ creates a two-sided curve: its positive-time part is simple and contained in the first quadrant, while its negative time part forms a simple closed curve rooted at the origin, satisfying the reflection property stated in Lemma~\autoref{lem:100}-(\ref{100.2}). We are now ready to prove that the increment of the left hull after time $t_{\ast}=4\pi/\im(c^2)$ is solely given by a positive time curve. 
\begin{lem}\label{lem:left_hull_increment}
    For each $c  \in \Omega_0$, denote by $\gamma_c$ the two-sided pioneer curve created by $\lambda_c$ up to a positive time $T$ and let $t_{\ast}=4\pi/\im(c^2)$. For all $s \geq 0$, we have $L_{t_{\ast}+s} = \gamma_c[-t_{\ast},t_{\ast}+s]$, where the positive time part is a pioneer curve.
\end{lem}
\begin{proof}
    We already know that up to time $t_{\ast}$, the left hull is given by a two-sided curve. It suffices to prove that for each $s \geq 0$ the increment $L_{t_{\ast}+s} \setminus L_{t_{\ast}}$ only contains its positive time part and the idea is to apply a self-similarity up to translation argument of the complex linear driver. By the concatenation principle~(\ref{concat:2}) the increment of the left hull is given by
    \begin{equation}\label{eq:increment}
        L_{t_{\ast}+s} \setminus L_{t_{\ast}} = g^{-1}_{t_{\ast}}\left((ct_{\ast}+L_s\right) \setminus R_{t_{\ast}}).
    \end{equation}
    Using the duality principle~(\ref{duality:2}) we have $R_{t_{\ast}} = ct_{\ast}+iL_{t_{\ast}}(\widetilde{\lambda}_c)$. Furthermore by $c \in \Omega_0$ the driver satisfies $\widetilde{\lambda}_c = -\lambda^{\ast}_c$, so that we can rewrite the increment as
    \[
    L_{t_{\ast+s}} \setminus L_{t_{\ast}} = g^{-1}_{t_{\ast}}\left(ct_{\ast} + \big(L_s \setminus -iL^{\ast}_{t_{\ast}}\big)\right).
    \]
    First consider $0 < s \leq 4\pi/\im(c^2)$, then by the reflection property stated in Lemma~\ref{lem:100}, we have $L_s \setminus -iL^{\ast}_{t_{\ast}} = \gamma_c(0,s]$, so that
    \[
    L_{t_{\ast}+s} \setminus L_{t_{\ast}} = g^{-1}_{t_{\ast}}\left(ct_{\ast}+\gamma^+_c(0,s]\right) = g^{-1}_{t_{\ast}}\left(L^{+}_{s}(\lambda_c(t_{\ast}+\cdot)) \setminus \{\lambda_c(t_{\ast})\}\right).
    \]
    Since we want to show that the left hand side is given by a positive time curve take $\gamma_c(t_{\ast}+s):=g^{-1}_{t_{\ast}}(\gamma^{\lambda(t_{\ast}+s)}(s))$. Then by the concatenation property~(\ref{concat:0}) we have
    \[
    \gamma_c(t_{\ast}+s)=g^{-1}_{t_{\ast}}(\gamma^{\lambda(t_{\ast}+s)}(s)) = g^{-1}_{t_{\ast}}\left(\lim_{\epsilon \downarrow 0} g^{-1}_{t_{\ast},t_{\ast}+s}(i\epsilon +\lambda_c(t_{\ast}+s)\right) = \lim_{\epsilon \downarrow 0} g^{-1}_{t_{\ast}+s}(i\epsilon+\lambda_c(t_{\ast}+s)).
    \]
    So each time instance $t_{\ast}+s$ on the curve is given by a top-frontier limit of the driving function $\lambda_c$ and it follows that the increment of the left hull $L_{t_{\ast}+s} \setminus L_{t_{\ast}}$ for $0 \leq s \leq 4\pi/\im(c^2)$ is given by a positive time curve in the sense of Definition~\ref{Def:Two-sidedcurve}. The general case for an arbitrary time increment now follows by using a self-similarity argument.

    Fix $t>0$, which we rewrite as $t=kt_{\ast}+s$ for some $k \in \mathbb N$ and $0 \leq s \leq t_{\ast}$. Observe that since $L_s \subseteq L_t$ for each $s \leq t$, we can rewrite the increment as the union of disjoint sets given by
    \begin{equation}\label{eq:union}
        L_{kt_{\ast}+s} \setminus L_{t_{\ast}} = \dot\cup^{k}_{j=2} \left(L_{jt_{\ast}+s} \setminus L_{(j-1)t_{\ast}}\,\right).
    \end{equation}
    We now apply a self-similarity argument to show that for each $2 \leq j \leq k$, the increment is given by a positive time curve and consequently glue them all together.

    Using the concatenation principle~(\ref{concat:3}) for each $2 \leq j \leq k$ in the above union, we have
    \[
    L_{jt_\ast+s} \setminus L_{(j-1)t_\ast} = g^{-1}_{(j-1)t{\ast}}\left(L_{t_\ast+s} \setminus R_{jt_\ast}\right).
    \]
    Observe that using the same duality arguments as above, we have $R_{jt_\ast}=-iL^{\ast}_{jt_\ast}$. Now consider the decomposition $L_{t_\ast+s} \setminus R_{jt_\ast} = L_{t_\ast} \setminus R_{jt_\ast} \cup ( \left(L_{t_\ast+s} \setminus L_{t_\ast} \right) \setminus R_{jt_\ast} \big)$. By Lemma\autoref{lem:100} the first term is only given by a positive time pioneer curve up to time $t_{\ast}$. Using the above argument, the increment $L_{t_\ast+s} \setminus L_{t_\ast}$ is solely given by a positive time pioneer curve as well, so that we can also apply Lemma\autoref{lem:100}-(\ref{100.2}) to this part. Using the above decomposition, this proves that $L_{t_\ast+s} \setminus R_{jt_\ast}$ is only given by a positive time curve and using Lemma~\ref{lem:upper_part_simple} we conclude that the curve is simple. Repeating this argument consecutively for each $2 \leq j \leq k$ in combination with~(\ref{eq:union}) finishes the proof.
\end{proof}

We are now ready to give a proof of Theorem~\ref{item:classification_zero}.
\begin{proof}[Proof of Theorem~\ref{item:classification_zero}]
Fix $c \in \Omega_0$, then by Lemma~\ref{lem:bound_tau} the complex linear driver $\lambda_c$ creates a left hull given by a two-sided curve in the sense of Definition~\ref{Def:Two-sided-pioneercurve} up to at least time $4\pi/\im(c^2)$, where the curve is pioneer for all $t < 4\pi/\im(c^2)$. This proves Theorem~\ref{subitem_zero:a}. By Lemma~\ref{lem:upper_part_simple}, the positive time part is simple and contained in the upper half-plane. By the reflection property stated in Lemma~\ref{lem:100} and its proof the negative time part up to $4\pi/\im(c^2)$ forms a Jordan curve traversing the fourth quadrant, which is rooted at the origin. The point $2/c$ is disconnected from infinity by the Jordan curve by the Lemma~\ref{lem:branch_cut_calc} but is never added to the left hull due to Lemma~\ref{lem:2/c}. This finishes Theorem~\ref{subitem_zero:b}. Applying Lemma~\ref{lem:left_hull_increment} proves Theorem~\ref{subitem_zero:c}.
\end{proof}
This finishes the proof of Theorem~\ref{thm:classification_linear_complete}.

\subsection{Optimal Hölder constant}\label{Section:5.3}

Recall that in Proposition~\ref{Thm:Tran's Theorem}, it was shown that if $\|\lambda\|_{1/2} \leq \sigma$, where $\sigma \in (1/3,4)$ is not explicitly known, then $\lambda$ creates a left hull given by a two-sided quasiarc. Recently in~\cite{lind2022phase}, by considering drivers of the form $U_c(t):=c\sqrt{1-t}$, it was proven that $\sigma < 3.772$. Using our classification result Theorem~\ref{item:classification_zero}, we are now improving on the bounds of this constant by proving a sharper upper bound for $\sigma$. We believe this upper bound to be non-optimal.
\begin{cor}\label{lem:optimal_sigma}
    For each $c \in \Omega_0$, denote by $\gamma_c$ the two-sided curve created up to time $t_{\ast}=4\pi/\im(c^2)$. For each $c \in \Omega_0$, we have $\|\lambda_c\big\vert_{[0,t_{\ast}]}\|_{1/2}=2\sqrt{\pi}$. Consequently, the constant $\sigma$ in Theorem~\ref{Thm:Tran's Theorem} satisfies the following bounds
    \[
    1/3 < \sigma < 2 \sqrt{\pi} \approx 3.544.
    \]
\end{cor}
\begin{proof}
    Fix $c \in \Omega_0$, by Theorem~\ref{subitem_zero:b}, the first time at which the two-sided curve $\gamma_c$ is non-simple is given by $t_{\ast}=4\pi/\im(c^2)$. 
    For $T \geq 0$ a quick algebraic manipulation gives
    \begin{equation}\label{eq:100}
         \|\lambda\vert_{[0,T]}\|_{1/2} = \sup_{s,t \in [0,T]} \frac{|\lambda_c(t)-\lambda_c(s)|}{|t-s|^{1/2}} = |c| \, \sup_{t,s \in [0,T]} |t-s|^{1/2} = |c| T^{1/2}.  
    \end{equation}
    Evaluating the above at time $T=t_{\ast}$ gives $\|\lambda_c\vert_{[0,t_{\ast}]}\|_{1/2} = |c| t_{\ast}^{1/2}$ and using $c_0 \in \Omega_0$ gives
    \[
    \|\lambda_c\vert_{[0,t_{\ast}]}\|_{1/2} = |c| \, \sqrt{t_{\ast}} = 2 \sqrt{\pi} \approx 3.5444.
    \]
    Observe that we just have proven that $\|\lambda_c\vert_{[0,t_{\ast}]}\|_{1/2}=2\sqrt{\pi}$ is independent of the choice of $c \in \Omega_0$ and this finishes the proof.
\end{proof}
\section{Open questions}\label{Sec:Section_6}

In this section, we pose some additional questions concerning complex-driven Loewner chains. Our questions mostly concern the deterministic setup, and we refer to \cite[Section~5]{gwynne2023loewner} for further open questions on complex-driven Schramm-Loewner evolutions. \newline

\textbf{Question 1.} Is it possible to encode each simple strictly increasing two-sided curve up to rotation via a complex driving function ? 
\vspace{0.3cm}

Recall the following classical fact from~\cite[Section~4.1]{lawler2008conformally}, which allows us to encode simple chords in the upper half-plane by continuous real-valued driving functions:
Let $\gamma \colon [0,\infty) \to \overline{\mathbb{H}}$ be a simple curve emanating from the real line. Reparameterize $\gamma$ by half-plane capacity so that $\text{hcap}(\gamma_t)=2t$ for each $t \geq 0$. For each $t \ge 0$, there exists a unique conformal map $g_t \,\colon \, \mathbb H \setminus \gamma_t \to \mathbb H$ with expansion around infinity given by
\[
g_t(z) = z + \frac{2t}{z} + O(|z|^{-2}), \qquad \text{ as } z \to \infty .
\]  
Moreover, each $g_t \, \colon \, \mathbb H \setminus \gamma_t \to \mathbb H$ extends continuously to the tip $\gamma(t)$ by mapping it to a point on the real axis. As a result, one can define the driver by setting $\lambda(t) := g_t(\gamma(t))$  for each $t \geq 0$ and prove that $\lambda$ is a continuous function. Having the continuous driver at hand, it can consequently be shown that for each $z \in \mathbb{H}$, the process $(g_t(z))_{0 \le t < T_z}$ solves the chordal Loewner equation given by
\[
\partial_t g_t(z) = \frac{2}{g_t(z) - \lambda(t)}, \qquad g_0(z) = z ,
\]  
where $T_z := \inf\{ t \ge 0 \colon \gamma(t) = z \}$. Thus, a simple chord in the upper half plane parametrized by half-plane capacity can be encoded as a continuous, real-valued driving function via the Loewner equation. A natural question is whether this framework extends to simple, strictly increasing, two-sided curves in the complex plane.\newline

Specifically, let $\gamma \colon (-\infty, \infty) \to \mathbb{C}$ be such a curve. Does there exist a parameterization such that, after reparameterizing $\gamma$, one can construct a suitable family of maps $(g_t)_{t \ge 0}$ that encodes the reparameterized curve in a similar way as above?\newline

Note, that we add the additional condition of "up to rotation" since we believe that it is not possible for complex driven Loewner chains to create the two-sided curve given by $\gamma_t :=[-t,t]$ for $t \geq 0$, or a two-sided line segment, where one of the ends stays fixed for a positive amount of time.

\vspace{0.3cm}

\textbf{Question 2.} In Theorem~\ref{thm:classification_linear_complete} we have seen that it is possible that a complex-valued driver creates a two-sided curve with a two-connected complement. In the spirit of Question~1, one may ask: If there exists a parametrization that allows every strictly increasing two-sided curve in the plane to be encoded by a complex-valued driver, does the same hold for self-touching curves with multiply connected complements, for instance a Jordan curve?

\vspace{0.3cm}

\textbf{Question 3.} Generalizations of the Loewner energy to the complex-driven case.\vspace{0.3cm}

We quickly review the Loewner energy in the real-driven chordal case, which was independently introduced by~\cite{dubedat2007commutation,friz2017existence,wang2019loewner}. Denote by $C_0([0,T])$ the set of continuous functions with $\lambda(0)=0$, endowed with the $L^{\infty}$-norm. A function is called \emph{absolutely continuous}, if there exists a function $h \in L^1\left([0,T]\right)$, such that $\int^{t}_0 h(s) ds = \lambda(t)$ for each $t \geq 0$. For each $0 \leq t \leq T$, using the duality property~(\ref{duality:2}) define the \emph{Loewner energy of the left (resp. right) hull up to time $t$} by setting
\[
\mathcal{E}_t(L):= \int^{t}_0 |\dot \lambda(s)|^2 \, ds, \qquad \text{ and } \qquad \mathcal{E}_t(R):=\int^{t}_0|i\dot \lambda(t-s)|^2 \, ds,
\]
if $\lambda_t$ is absolutely continuous and infinity otherwise. Using the translation property of the left~(\ref{trans:0}) (resp. right~(\ref{trans:1})) hull, the energy of both hulls is translation invariant. Precisely speaking for each $a \in \mathbb C$ and each $t \geq 0$, we have $\mathcal{E}_t(L+a)=\mathcal{E}_t(L)$ (resp. $\mathcal{E}_t(R+a)=\mathcal{E}_t(R)$), since shifting the hull by $a$ results in shifting the driver by the same constant. Moreover, using~(\ref{scaling:0}) and~(\ref{scaling:1}), we can establish the scale invariance of the energy. For $a>0$ and each $t \geq 0$, the left hull $aL_t$ is driven by $a\lambda(\cdot/a^2)$ up to time $a^2t$, so performing a quick algebraic manipulation it is not hard to see that $\mathcal{E}_t(aL)=\mathcal{E}_{t}(L)$ and $\mathcal{E}_t(aR)=\mathcal{E}_{t}(R)$. Most importantly, using the concatenation principle~(\ref{concat:1}) we can prove that
\[
\mathcal{E}_{t+s}(L) = \mathcal{E}_{t,t+s}(L) + \mathcal{E}_{t}(R) \qquad \text{ and } \qquad \mathcal{E}_{T}(L) = \mathcal{E}_{T}(R),  
\]
where $0 \leq t \leq t+s \leq T$. In other words, under the assumption that a left hull $L$ has finite energy up to time $T$, the energy is split between the left and the right hull, such that the mapping out function $g_t \, \colon \, \mathbb C \setminus L_t \to \mathbb C \setminus R_t$ transfers the energy from the left to the right hull while conserving the total mass of energy. Now under the assumption that we can encode a two-sided chord $\gamma \, \colon \, [-T,T] \to \mathbb C$ via the a complex driver $\lambda \in C^{0}\left([0,\infty);\mathbb C\right)$ as described in Question 1, what can we say about for example the space of finite energy curves? Observe that using~\cite[Theorem~2]{lind2022phase} or Theorem~\ref{thm:classification_linear_complete}, there exist smooth driving functions with finite energy curves that are non-simple. Consequently, the space of finite energy curves already differs from the real-driven chordal case.

\vspace{0.3cm}
\textbf{Question 4.} What is the optimal constant $\sigma$ in Theorem~\ref{Thm:Tran's Theorem}?

\appendix

\section{}
We primarily provide references to theorems in the context of complex-driven Loewner chains. Lemma~\ref{lem:calculation_asympt} performs a calculation required for the proof of Lemma~\ref{lem:3}.
\begin{lem}\label{lem:regularity_mappingout}
    Let $\lambda \colon [0,\infty) \to \C$ be continuous and let $(g_t)_{t \geq 0}$ be the mapping out functions driven by $\lambda$. Let $z \in \C$ such that $T_z \leq C < \infty$ then 
    \[
    \lim_{t \uparrow T_z} |g_t(z)-\lambda(t)| = 0.
    \]
\end{lem}
\begin{proof}
    See~\cite[Lemma~2.5]{gwynne2023loewner}.
\end{proof}

\begin{prop}[Jordan curve Theorem]\label{prop:Jordan_curve}
If $\gamma$ is a simple closed curve in $\C$, then $\C \setminus \gamma$ contains of one bounded and one unbounded component, each of which has $\gamma$ as its boundary. 
\end{prop}
\begin{proof}
    See~\cite[Theorem~1.3]{conway2012functions}.
\end{proof}

\begin{lem}\label{lem:9}
Let $\lambda \colon [0, \infty) \to \mathbb{C}$ be a continuous function and let $(f_t)_{t \geq 0}$ be the centered Loewner chain driven by $\lambda$. Assume that there is a simple two-sided curve $\eta : \mathbb{R} \to \mathbb{C}$ such that the left hull satisfies $L_t = \eta([-t,t])$ for each $t \geq 0$ and for each $t \geq 0$, the right hull is equal to the image of some simple curve. For each $t > 0$ and each $z \in \gamma([-t,t])$, let $z^-$ (resp. $z^+$) be the prime end of $\mathbb{C} \setminus \gamma([-t,t])$ corresponding to the left (resp. right) side of $\eta([-t,t])$. Then for each $t \geq 0$ and each $s \in [0,t]$,
\[
f_t(\eta(s)^+) = f_t(\eta(-s)^+) \quad \text{ and } f_t(\eta(s)^-) = f_t(\eta(-s)^-).
\]
\end{lem}
\begin{proof}
    See~\cite[Lemma~2.6]{gwynne2023loewner}.
\end{proof}

\begin{lem}\label{lem:Union_connected}
    Let $(M,d)$ be a metric space. Then if $\mathcal{F}$ is a collection of connected subsets of $M$ that have non-empty intersection then $\mathcal{F}$ is also connected. 
\end{lem}
\begin{proof}
    See~[Theorem~23.3]\cite{munkres2013topology}.
\end{proof}

\begin{lem}\label{lem:calculation_asympt}
    For any $c \in \Omega_{+} \cup \Omega_{-}$, define $\lambda_c \, \colon \, [0,\infty) \to \mathbb C$ by setting $\lambda_c(t)=ct$. Denote by $\gamma_c$ the two-sided pioneer curve created by $\lambda_c$ and consider the following expressions
    \[
    \mathcal{L}_{\pm}(t):=\frac{\log(|2-c\gamma_c(\pm t)|}{t}, \qquad \mathcal{R}_{\pm}(t):= \frac{\re(\gamma(\pm t))}{t}, \qquad \text{ and } \qquad \mathcal{I}_{\pm}(t):= \frac{\im(\gamma(\pm t))}{t}.
    \]
    Then the following statements are true.
    \begin{enumerate}
        \item\label{item:calculation_asympt_1} For $c \in \Omega_{+}$, we have $\mathcal{L}_{\pm}(t) \to 0$ while $\mathcal{R}_{\pm}(t)$ and $\mathcal{I}_{\pm}(t)$ converge to non-zero constants, as $t \uparrow \infty$.
        \item\label{item:calculation_asympt_2} For $c \in \Omega_{-}$, we have $\im(\gamma_c(-t)) \leq 0$ for each $t \geq 0$.
    \end{enumerate}
\end{lem}
\begin{proof}
    Fix $c \in \Omega_{+}$ and recall that we already know by Lemma~\ref{lem:construction_hull_simple} that $\lambda_c$ creates a two-sided pioneer curve in the sense of Definition~\ref{Def:Two-sided-pioneercurve}. Recall, that the positive (resp. negative) time part of the curve satisfy the real and imaginary part of the pioneer equation
    \begin{align}
    \re(c) \re(\gamma(\pm t)) - \im(\gamma(\pm t)) \im(c) + 2 \log\left(|2-c\gamma(\pm t)|\right) &= 2\log(2) + \re(c^2) t \tag{RE},\\
    \re(c) \im(\gamma(\pm t)) + \im(c) \re(\gamma(\pm t)) + 2 \text{Arg}\left(2-c \gamma(\pm t)\right)  &= \im(c^2) t\tag{IM}.
    \end{align} 
    Consider the real part of the pioneer equation and divide everything by $t$, then we are interested in taking the limits of the following expression
    \[
    \re(c) \mathcal{R}_{\pm}(t) - \im(c) \mathcal{I}_{\pm}(t) + 2 \mathcal{L}_{\pm}(t) = 2 \log(2) + \re(c^2).
    \]
    Now, as $t \uparrow \infty$, we see that the limit on the right hand side exists and is strictly positive, so the limit on the left hand side exists and must also be strictly positive. We first prove that $\mathcal{L}_{\pm}(t) \to 0$ as $t \uparrow \infty$. Suppose that
    there exists $C_1 \neq 0$, such that $\mathcal{L}_{\pm}(t) \to C_1$ as $t \uparrow \infty$. Then by the properties of the logarithm, it must be that either the first term $\mathcal{R}_{\pm}(t)$ or the second term $\mathcal{I}_{\pm}(t)$ diverges as $t \uparrow \infty$, contradicting the fact that the right hand side is finite. Consequently, $\mathcal{L}_{\pm}(t) \to 0$ as $t \uparrow \infty$, and it is left to prove that $\mathcal{R}_{\pm}(t) \to C_1$ and $\mathcal{I}_{\pm}(t) \to C_2$ as $t \uparrow \infty$, where $C_1,C_2 \neq 0$. Since the right hand side is strictly positive, both $C_1,C_2$ cannot be zero simultaneously. Suppose that $\mathcal{R}_{\pm}(t) \to 0$ as $t \uparrow \infty$, then $\mathcal{I}_{\pm}(t) \to -\re(c^2)/\im(c)$ as $t \uparrow \infty$, using the imaginary part of the pioneer equation~(\ref{eq:IM}), we also get $\mathcal{I}_{\pm}(t) \to \im(c^2)/\re(c)$ as $t \uparrow \infty$. Since $c \in \Omega_{+}$, using the uniqueness of limit shows that this is impossible. A similar strategy works if one supposes that $\mathcal{I}_{\pm}(t) \to 0$ as $t \uparrow \infty$ and this finishes the proof.

    For our second claim~(\ref{item:calculation_asympt_2}), write $c=a+ib$, where $a,b > 0$ with $b^2 >a^2$ by $c \in \Omega_{-}$. Recall, that by the remarks of Definition~\ref{Def:Left/Right_hulls}, we know that $\re(\gamma_c(-t)) \geq 0$.
    Observe that the right hand side diverges to $-\infty$ as $t \uparrow \infty$. Since $\im(c) > 0$, we need to analyze the limits of $\mathcal{L}_{-}(t)$ and $\mathcal{I}_{-}(t)$ as $t \uparrow \infty$. If we can rule out that the second term diverges, then we are done.

    Using similar techniques as given in Lemma~\ref{lem:upper_part_simple}, it is not hard to see that $\im(\gamma_c(-t))<0$ for small times. If we can rule out that the curve crosses the real line, then we are done. Denote by $t_{\ast}:= \inf \{0 \leq t \leq T \, \colon \, \im(\gamma_c(-t))=0\}$ the first the curve crosses the real line after entering the fourth quadrant and let us define $F \, \colon \, [0,\infty) \to \mathbb R$ by setting $F(x):=\re(c)x+\log(|2-cx|^2/4)$, which equals the left hand side of the pioneer equation if the curve at a given time instant is purely real. Since the right hand side is strictly negative, we are only interested in the values of $F$ that are potentially negative. As a result, we can rule out points on the real line, where the $x \mapsto \log(|2-cx|^2/4)$ is positive, and a quick calculation shows that this is the case for $4a/\left(a^2+b^2\right)<x$. So if $t_{\star}$ is finite, then $0<\gamma(t_{\ast}) \leq 4a/\left(a^2+b^2\right)$. But now observe that $F(x)=0$ and $F^{\prime}(x) >0$ for $0 \leq x \leq 4a/\left(a^2+b^2\right)$, so that at time $t_{\ast}$ the left hand side is positive but the right hand side is strictly negative.
\end{proof}

\printbibliography

@article{schramm2000scaling,
  title={Scaling limits of loop-erased random walks and uniform spanning trees},
  author={Schramm, Oded},
  journal={Israel Journal of Mathematics},
  volume={118},
  number={1},
  pages={221--288},
  year={2000},
  publisher={Springer}
}

@article{dubedat2007commutation,
  title={Commutation relations for Schramm-Loewner evolutions},
  author={Dub{\'e}dat, Julien},
  journal={Communications on Pure and Applied Mathematics: A Journal Issued by the Courant Institute of Mathematical Sciences},
  volume={60},
  number={12},
  pages={1792--1847},
  year={2007},
  publisher={Wiley Online Library}
}

@phdthesis{wang2019loewner,
  title={On the Loewner energy of simple planar curves},
  author={Wang, Yilin},
  year={2019},
  school={ETH Zurich}
}

@article{kennedy2007fast,
  title={A fast algorithm for simulating the chordal Schramm--Loewner evolution},
  author={Kennedy, Tom},
  journal={Journal of Statistical Physics},
  volume={128},
  number={5},
  pages={1125--1137},
  year={2007},
  publisher={Springer}
}

@article{de1985proof,
  title={A proof of the Bieberbach conjecture},
  author={De Branges, Louis},
  journal={Acta Mathematica},
  volume={154},
  number={1},
  pages={137--152},
  year={1985},
  publisher={Kluwer Academic Publishers Dordrecht}
}

@book{kemppainen2017schramm,
  title={Schramm--loewner evolution},
  author={Kemppainen, Antti and Kemppainen, Antti},
  year={2017},
  publisher={Springer}
}

@article{bieberbach1916uber,
  title={Uber die Koeffizienten derjenigen Potenzreihen, welche eine schlichte Abbildung des Einheitskreises vermitteln},
  author={Bieberbach, Ludwig},
  journal={Sitzungsberichte Preussische Akademie der Wissenschaften},
  volume={138},
  pages={940--955},
  year={1916}
}

@article{camia2006two,
  title={Two-dimensional critical percolation: the full scaling limit},
  author={Camia, Federico and Newman, Charles M},
  journal={Communications in Mathematical Physics},
  volume={268},
  number={1},
  pages={1--38},
  year={2006},
  publisher={Springer}
}

@article{smirnov2001critical,
  title={Critical percolation in the plane: conformal invariance, Cardy's formula, scaling limits},
  author={Smirnov, Stanislav},
  journal={Comptes Rendus de l'Acad{\'e}mie des Sciences-Series I-Mathematics},
  volume={333},
  number={3},
  pages={239--244},
  year={2001},
  publisher={Elsevier}
}

@article{lowner1923untersuchungen,
  title={Untersuchungen {\"u}ber schlichte konforme Abbildungen des Einheitskreises. I},
  author={L{\"o}wner, Karl},
  journal={Mathematische Annalen},
  volume={89},
  number={1},
  pages={103--121},
  year={1923},
  publisher={Springer}
}

@article{tran2017loewner,
  title={Loewner equation driven by complex-valued functions},
  author={Tran, Huy},
  journal={arXiv preprint arXiv:1707.01023},
  year={2017}
}

@unpublished{RS,
    author = {Steffen Rohde and Oded Schramm},
    title = {Unpublished},
}

@article{rohde2018loewner,
  title={The Loewner equation and Lipschitz graphs},
  author={Rohde, Steffen and Tran, Huy and Zinsmeister, Michel},
  journal={Revista matem{\'a}tica iberoamericana},
  volume={34},
  number={2},
  pages={937--948},
  year={2018}
}

@article{gwynne2023loewner,
  title={Loewner evolution driven by complex Brownian motion},
  author={Gwynne, Ewain and Pfeffer, Joshua and Park, Minjae},
  journal={The Annals of Probability},
  volume={51},
  number={6},
  pages={2086--2130},
  year={2023},
  publisher={Institute of Mathematical Statistics}
}

@article{lind2022phase,
  title={Phase transition for a family of complex-driven Loewner hulls},
  author={Lind, Joan and Utley, Jeffrey},
  journal={Involve, a Journal of Mathematics},
  volume={15},
  number={3},
  pages={447--474},
  year={2022},
  publisher={Mathematical Sciences Publishers}
}

@article{kager2004exact,
  title={Exact solutions for Loewner evolutions},
  author={Kager, Wouter and Nienhuis, Bernard and Kadanoff, Leo P},
  journal={Journal of statistical physics},
  volume={115},
  pages={805--822},
  year={2004},
  publisher={Springer}
}

@book{lawler2008conformally,
  title={Conformally invariant processes in the plane},
  author={Lawler, Gregory F},
  number={114},
  year={2008},
  publisher={American Mathematical Soc.}
}

@inproceedings{lind2005sharp,
  title={A sharp condition for the Loewner equation to generate slits.},
  author={Lind, Joan R},
  booktitle={Annales Academiae Scientiarum Fennicae. Mathematica},
  volume={30},
  number={1},
  pages={143--158},
  year={2005},
  organization={Finnish Academy of Science and Letters, Helsinki; Finnish Society of~…}
}

@article{sheffield2009exploration,
  title={Exploration trees and conformal loop ensembles},
  author={Sheffield, Scott},
  year={2009}
}

@article{lawler2004conformal,
  title={Conformal invariance of planar loop-erased random walks and uniform spanning trees},
  author={Lawler, Gregory F and Schramm, Oded and Werner, Wendelin},
  journal={The Annals of Probability},
  volume={32},
  number={1B},
  pages={939--995},
  year={2004},
  publisher={Institute of Mathematical Statistics}
}

@article{friz2017existence,
  title={On the existence of SLE trace: finite energy drivers and non-constant $\kappa$},
  author={Friz, Peter K and Shekhar, Atul},
  journal={Probability theory and related fields},
  volume={169},
  number={1},
  pages={353--376},
  year={2017},
  publisher={Springer}
}

@article{schramm2013contour,
  title={A contour line of the continuum Gaussian free field},
  author={Schramm, Oded and Sheffield, Scott},
  journal={Probability Theory and Related Fields},
  volume={157},
  number={1},
  pages={47--80},
  year={2013},
  publisher={Springer}
}

@article{chelkak2014convergence,
  title={Convergence of Ising interfaces to Schramm's SLE curves},
  author={Chelkak, Dmitry and Duminil-Copin, Hugo and Hongler, Cl{\'e}ment and Kemppainen, Antti and Smirnov, Stanislav},
  journal={Comptes Rendus. Math{\'e}matique},
  volume={352},
  number={2},
  pages={157--161},
  year={2014}
}

@article{rushkin2006stochastic,
  title={Stochastic Loewner evolution driven by L{\'e}vy processes},
  author={Rushkin, Ilia and Oikonomou, Panagiotis and Kadanoff, Leo P and Gruzberg, Ilya A},
  journal={Journal of Statistical Mechanics: Theory and Experiment},
  volume={2006},
  number={01},
  pages={P01001},
  year={2006},
  publisher={IOP Publishing}
}

@article{smirnov2010conformal,
  title={Conformal invariance in random cluster models. I. Holmorphic fermions in the Ising model},
  author={Smirnov, Stanislav},
  journal={Annals of mathematics},
  pages={1435--1467},
  year={2010},
  publisher={JSTOR}
}

@article{pete2016conformally,
  title={A conformally invariant growth process of SLE excursions},
  author={Pete, G{\'a}bor and Wu, Hao},
  journal={arXiv preprint arXiv:1601.05713},
  year={2016}
}

@article{peltola2024loewner,
  title={Loewner traces driven by L{\'e}vy processes},
  author={Peltola, Eveliina and Schreuder, Anne},
  journal={arXiv preprint arXiv:2407.06144},
  year={2024}
}

@article{lind2010collisions,
  title={Collisions and spirals of Loewner traces},
  author={Lind, Joan and Marshall, Donald E and Rohde, Steffen},
  year={2010}
}

@article{marshall2005loewner,
  title={The Loewner differential equation and slit mappings},
  author={Marshall, Donald and Rohde, Steffen},
  journal={Journal of the American Mathematical Society},
  volume={18},
  number={4},
  pages={763--778},
  year={2005}
}

@book{munkres2013topology,
  title={Topology: Pearson New International Edition},
  author={Munkres, James R},
  year={2013},
  publisher={Pearson Higher Ed}
}

@book{conway2012functions,
  title={Functions of one complex variable II},
  author={Conway, John B},
  volume={159},
  year={2012},
  publisher={Springer Science \& Business Media}
}

\end{document}